\documentclass[12pt,a4paper]{amsart}
\usepackage{geometry, comment}
\usepackage{color}
\usepackage{graphicx}
\usepackage{amssymb}
\usepackage{amsrefs}

\def\la{\lambda}
\def\si{\sigma}

\def\de{\delta}
\newcommand{\Int}{\hbox{\rm int}}
\newcommand{\R}{\mathbb R}
\newcommand{\G}{\mathcal G}
\newcommand{\A}{\mathcal A}

\newcommand{\M}{\mathcal M}
\newcommand{\PP}{\mathfrak P}
\newcommand{\N}{\mathbb N}

\newcommand{\MM}{\mathfrak{M}}
\newcommand{\eps}{\varepsilon}

\def\bald{\begin{aligned}}\def\eald{\end{aligned}}
\DeclareMathOperator{\spt}{supp}

\newtheorem{Theorem}{Theorem}[section]
\newtheorem{Proposition}[Theorem]{Proposition}
\newtheorem{Lemma}[Theorem]{Lemma}
\newtheorem{Definition}[Theorem]{Definition}
\newtheorem{Corollary}[Theorem]{Corollary}

\newtheorem{Remark}[Theorem]{Remark}
\newtheorem{example}[Theorem]{Example}

 \title
 {The vanishing discount problem for Hamilton--Jacobi equations in the Euclidean space}

\thanks{   {\bf Key words.} vanishing discount problem,  Hamilton--Jacobi equations, ergodic problem.
\\{\bf AMS subject classifications.} 35B40, 35F21, 37J99, 49L25\\
The work of HI was partially supported by the JSPS grants: KAKENHI
\#16H03948, \#18H00833 and the NSF Grant No. 1440140
and the work of AS was partially supported by Fondo Ateneo 2017--Universit\`{a} di Roma "La Sapienza" and the NSF Grant No. 1440140}

\author[Ishii]{Hitoshi Ishii$^*$ }

 \address{Institute for Mathematics and Computer Science\\ Tsuda  University\\170
 2-1-1 Tsuda, Kodaira\\ Tokyo 187-8577\\Japan}
\email{}\email{hitoshi.ishii@waseda.jp}
\thanks{${}^*$ Corresponding author}

 \author[Siconolfi]{Antonio Siconolfi}

\address{Department of Mathematics \\
                 Universit\`a degli Studi di Roma ``La Sapienza''\\
         Piazzale Aldo Moro 5  \\ 00185 Roma\\
                   Italy.}
\email{siconolf@mat.uniroma1.it}

\begin{document}

 \begin{abstract} We study the asymptotic behavior of the solutions to a family of
discounted Hamilton--Jacobi equations,  posed in $\R^N$, when the discount factor goes to zero.
The ambient space being noncompact,   we introduce an assumption  implying that the Aubry set is compact and there is no degeneracy at infinity.
Our approach is to deal not with  a single Hamiltonian and Lagrangian but  with  the whole space of generalized Lagrangians,
and then  to define via duality  minimizing measures associated to both the corresponding ergodic and discounted equations.
The asymptotic result follows from  convergence properties of these measures with respect to the narrow topology.
We use as duality tool  a separation theorem in locally convex Hausdorff  spaces,
we use  the strict topology in the space of the bounded generalized Lagrangians  as well.
\end{abstract}


 \maketitle
\tableofcontents

\section{Introduction}

We study  the asymptotic behavior, as the discount factor $\la > 0$ goes to $0$, of the   viscosity solutions  to the Hamilton--Jacobi
equations
\[\lambda \, u + H(x,Du) =c  \]
posed in the Euclidean space $\R^N$. Here  $c$  is the so--called critical value defined as
 \[ c= \inf \{ a \mid H=a \;\hbox{admits  subsolutions in $\R^N$}\}.\]
Under our assumption this quantity is actually finite  and is a minimum.

Our output provides  an extension   to the noncompact setting of the selection principle, first established  in the compact case  in \cite{DFIZ2016}.
It asserts   that the  whole  family of solutions of the discounted problems, which are uniquely solved if the ambient space is compact,
 converges to
  a  distinguished solution   of the  ergodic  limit equation
 \[ H(x,Du)=c.\]
The latter  has instead multiple solutions,   parametrized
by the Aubry set, denoted by $\mathcal A$,  which is, roughly speaking,  the set of points
where is concentrated the  obstruction of getting subsolutions to $H=a$, for $a < c$, see Appendix \ref{KAM}.

We assume the  Hamiltonian $H(x,p)$ from $\R^N \times \R^N$ to $\R$  to be  continuous in both arguments, and convex, coercive
in the momentum variable, locally uniformly in $x$. Since  $H$ can be modified for $p$ of large norm,
 without affecting the analysis, a  superlinear growth as $|p| \to + \infty$,
can be in addition postulated  without loss of generality, see Proposition \ref{square}.  A
 Lagrangian  denoted by $L$, can  be then defined via Fenchel transform.

We have one more key condition, see \eqref{A3}/\eqref{A3'}, to specifically deal with  the  lack of compactness
 of the ambient space.  It implies that the Aubry set is nonempty, compact and that the intrinsic distance associated to $H=c$,  see Appendix \ref{KAM},
  is equivalent to the Euclidean one at infinity. Loosely speaking, the latter condition means that there is no Aubry set at infinity.
  In the case where the Hamiltonian is of the form
\[ H(x,p) = |p| - f(x) \qquad\hbox{ with $f$ continuous potential}\]
this corresponds requiring  the infimum  of $f$ to be  not attained at infinity.

Under our assumption, due to the noncompactness, either of the discounted
equations does not anymore single out a unique solution, see example in Section \ref{DDPP},  and the Aubry set fails to be a uniqueness set for the critical equation.

This fact leads to single out  a special  type of solutions to the critical equation in $\R^N$, named weak KAM solutions,
 and defined as the  functions $u$ for which
\[ u(x)= \min \{u(y) + S_0(y,x) \mid y \in \mathcal A \} \qquad\hbox{for any $x$},\]
where $S_0$ is the intrinsic (semi)distance associated to the critical equation. In our setting they
are characterized among all the critical solutions, see Theorem \ref{superKAM},
by the property of being bounded from below.  By definition  $\mathcal A$ is
then  an  uniqueness set for the weak KAM solutions.

Regarding the discounted equations,  we consider the maximal solution obtained as   the
 pointwise supremum of  the family of all subsolutions. they possess, like the weak KAM solutions, the  crucial property
  of being bounded from below, see Section \ref{DDPP}.

Our main result can therefore be stated as follows:

\smallskip

{ \bf Theorem} { \em The whole family of maximal solutions to  the discounted equations  converges locally uniformly
to a distinguished  weak KAM  solution of  the limit ergodic equation.}

\smallskip

As in \cite{DFIZ2016},  we  derive the asymptotic  behavior of solutions from  weak convergence of suitable  associated   measures.
Our method however  is  rather different. The relevant measures are not defined as occupational measures on curves, and  we seldom employ
representation formulae  for solutions  or  properties of curves  in the space of state variable.

 Our approach  instead relies
on some  functional analysis and  appropriate  duality principles between spaces of generalized Lagrangians and spaces of measures.
We define in this way  minimizing measures, named after Mather,  associated to both the ergodic  and discounted equations.

 For the ergodic equation, they coincide with the classical Mather measures given when the Hamiltonian is in addition Tonelli  and the ambient space compact.
 We also recover the relevant property that the closure of the  union of the supports of such measures is an uniqueness set for the weak KAM solutions, see Section \ref{mather}.

Our procedure  is  close in spirit to  Evans interpretation of Mather theory
in terms of complementarity problems, see \cite{E1}, \cite{E2}, and also \cite{G2005}. We think that this alternative approach is interesting per se and can
handle to extend the asymptotic result to  more general setting, for instance in the case of fully nonlinear second order equations (see \cite{IMT2017} for
such generalizations).

The idea of performing some  duality between generalized Lagrangians and measures,  in order to study the asymptotic of solution to discounted equations,  has been introduced in \cite{IMT2017}. The authors  however use as duality tool the  Sion minimax Theorem, while we instead employ  a separation result for convex subsets in locally convex Hausdorff space, see Appendix \ref{strict}.

 It  implies that the normal cone at any element of the boundary of a convex set with nonempty interior has nonzero elements. We actually find the Mather measures as elements, up to change of sign, of the normal cone at $L$ of suitable convex sets
in the space of generalized Lagrangians.
We need  for this an appropriate topological frame.

We consider the space of bounded continuous functions
from $\R^{2N}$ to $\R^N$  equipped with
the so--called strict topology, see Appendix \ref{strict}. In this case a nice
 generalization  of Riesz representation theorem holds true, namely the    topological dual  is the space of   signed
Borel  measures with bounded variation with the narrow topology as
 corresponding weak star topology.

To implement our method,  some effort has been put into constructing
 convex subsets of the space of bounded generalized Lagrangians
with nonempty interior and  $L  \wedge M$ as boundary point, for  suitable constants $M$.
To this aim, we have preliminarily proved  some localization results for both the ergodic and discounted equations, see Section \ref{locale}
and Propositions \ref{ana}, \ref{barra}.
\bigskip

\section*{Acknowledgements} The first author thanks the Department of
Mathematics at the Sapienza University of Rome
for its financial support and warm hospitality while
his visits (May 2018 and May 2019). The authors thank
the Department of Mathematics at Nanjing University
for its financial support and kind hospitality while they were visiting there (July 2018)
and also thank
the Mathematical Sciences Research
Institute in Berkeley for its financial support and warm hospitality while
their visit (HI: October 2018, AS: September--December 2018).

\section{Setting}

Given $R >0$,  we denote by $B_R$ the open ball of $\R^N$ or $\R^N \times \R^N$ centered  at $0$ with radius $R$, we write insead
$B(x_0,R)$ if the center is at $x_0$. Given two elements
$x$, $y$ of $\R^N$, we write $x \cdot y$ to indicate their scalar product. For any subset $E$ of a topological space, we denote by $\overline E$,
$\Int \, E$, $\partial E$ its closure interior and boundary, respectively.  If $u$ is a locally Lipschitz continuous function from $\R^N$ to $\R$ we define
its (Clarke) generalized  gradient at some point $x$ via
\[\partial u(x) = \mathrm{co} \{\lim_i Du(x_i) \mid x_i \to x,\; x_i \;\hbox{differentiability points of $u$}\}\]
where $\mathrm{co}$ stands for convex hull.

We consider an Hamiltonian $H: \R^N \times \R^N \to \R$ satisfying the following conditions

\[\tag{A1}\label{A1}
\begin{minipage}{0.8\textwidth} $H\in C(\R^N\times\R^N).$
\end{minipage}
\]
\[
\tag{A2}\label{A2}
\begin{minipage}{0.8\textwidth}
$H$ is convex and coercive, that is, for any $x\in\R^N$, the function $H(x,\cdot)$ is convex in $\R^n$ and for any $R>0$,
\[
\lim_{R\to\infty}\inf\{H(x,p)\mid x\in B_R,\ p\in \R^N\setminus B_R\}=+\infty.
\]
\end{minipage}
\]
\[\tag{A3}\label{A3} \begin{minipage}{0.8\textwidth}
There exists $\eps >0 $ such that
\[
  \limsup_{|x| \to + \infty} \,  \max_{p \in B_\eps} H(x,p) < \max_{x \in \R^N}  \, \min_{p\in \R^N} H(x,p)
\]
\end{minipage}
\]
\smallskip

We further consider the discount problem for the Hamilton-Jacobi equation
\begin{equation}  \tag{DP}\label{DP}
\la u+H(x,Du)= c \ \ \text{ in }\ \R^N,
\end{equation}
 and the associated  ergodic  problem
\[ H(x,Du)=c \ \ \ \text{ in }\ \R^N, \]
where  $\la>0$ is a given constant, and
\begin{equation}\label{defc}
  c= \inf \{ a \mid H(x,Du)=a \;\hbox{admits  subsolutions in $\R^N$}\}
\end{equation}
is the so--called critical value of $H$. We will show that in our setting it is finite and is actually a minimum.  Here and in what follows,
 the terms  solutions,  subsolutions, and
  supersolutions of Hamilton-Jacobi equations must be understood in the viscosity sense. We henceforth suppress the adjective word
{\em viscosity}. We record for later use a weaker version of \eqref{A3}:

\[\tag{A3'}\label{A3'} \begin{minipage}{0.8\textwidth}
There exists $\eps >0 $ such that
\[ \limsup_{|x| \to + \infty} \,  \max_{p \in B_\eps} H(x,p) < c. \]
\end{minipage}\]
It is clear that the critical  value is greater than or equal to the right hand--side of \eqref{A3}. The advantage of the formulation
\eqref{A3} is that that the quantity in the  right hand--side is {\em observable} for any given Hamiltonian while  the critical value could
 be in general not easy to compute.

We can assume by normalization that $c=0$ and, consequently, the ergodic  problem
 is stated as
\[ \tag{EP}\label{EP} H[u]=0 \ \ \text{ in }\ \R^n.\]
Here we mean that the problems involving the Hamiltonian $H$ are normalized
so as to $c$ when $H$ is replaced by $H-c$.  Note that if $H$ satisfies (A1)--(A3) and
$c\in\R$, then $H-c$ satisfies (A1)--(A3) as well.
We interpret \eqref{DP} as an approximation procedure for
\eqref{EP} when $\la$ is sent to $0$.

\smallskip

The first author  studied, under assumptions to be compared with  (A1)--(A3), the large time
behavior of the Hamilton-Jacobi equation  in $\R^N$ in \cite{HI2008}.

\smallskip

Condition \eqref{A3}  implies:

\begin{Proposition}\label{prescorcia} Assume that $u$ is a subsolution of $H[u]=a$ for some $a \in \R$
and $K$ a compact subset of $\R^N$. Then there exists a   subsolution of the same equation, constant outside some compact  subset,
coinciding with $u$ on $K$.
\end{Proposition}
\begin{proof}  Due to \eqref{A3} and  $H=a$ admitting subsolutions, we have
\[ a > \limsup_{|x| \to + \infty} \max_{p \in B_\eps} H(x,p) \quad\hbox{for some $\eps >0$,}\]
 we can therefore  take a compact subset $C$ with
\[ \Int \, C \supset K \cup \left  \{ x \mid \max_{p \in B_\eps} H(x,p) \geq  a \right \} .\]
We set $\phi(x)= - \frac \eps 2 \, |x|$  and select  $b >0 $ with $\min_C (\phi + b) > \max_C u$.  Because of the definition of $C$, the function
\[v = \min \{ \phi +b, u\}\]
is a subsolution of $H=a$ with $v =u$ on $K$ and
\[\lim_{|x| \to + \infty} v = - \infty.\]
 The function
\[w_0= \max \{ v, \min_C u\},\]
satisfies the assertion.
\end{proof}

\bigskip

\section{Maximal subsolutions of \eqref{DP}}\label{DDPP}

The first result of the section is

\begin{Proposition}\label{prop4-2}
  The  family of subsolutions to \eqref{DP} is locally equibounded from above, when   $\lambda$ varies in $(0,+\infty)$.
\end{Proposition}

\smallskip

A lemma is preliminary

\smallskip

\begin{Lemma} \label{lem1-3}
For each $R>0$, there exist a constant $C_R>0$ and a function
$\psi_R\in C^1(B_R)$ such that
\[H[\psi_R]> -C_R \ \ \ \text{ in }\ B_R,\quad \text{and} \quad
\lim_{|x|\to R^-}\psi_R(x)=+\infty. \]
\end{Lemma}

\begin{proof}  Fix $R>0$ and choose a function $\psi_R\in C^1(\R^N)$
so that
\[
\lim_{|x|\to R^-}\psi_R(x)=+\infty \quad\hbox{and}\quad \lim_{|x|\to R^-}|D\psi_R(x)|=+\infty.
\]
Observe that
\[
x\mapsto H(x,D\psi_R(x))
\]
is continuous on $B_R$ and that
\[
\lim_{|x|\to R^-}H(x,D\psi_R(x))=+\infty.
\]
It is now obvious that
\[
x\mapsto H(x,D\psi_R(x))
\]
has a minimum in $B_R$. Thus, for some constant $C_R>0$,
\[H(x,D\psi_R(x))\geq -C_R  \qquad\hbox{ in }\ B_R. \]
\end{proof}

\smallskip

\begin{proof} [Proof of Proposition   \ref{prop4-2}]
Let $u$ be  any  subsolution of \eqref{DP}, for some $\lambda >0$. Fix  $R>0$.
According to Lemma \ref{lem1-3}, there are a function $\psi\in C^1(B_R)$ and a constant $ b>0$ such that
\[
H(x,D\psi(x))\geq -b \ \ \ \text{ for }\ x\in B_R \ \ \ \hbox{and} \ \ \
\lim_{|x|\to R^-}\psi(x)=+\infty.
\]
By adding a constant to $\psi$ if necessary, we may assume that
$\psi\geq 0$ in $B_R$.

Set
\[  v(x)= \psi(x)+ \lambda^{-1}  \, b  \qquad \hbox{ for } \; x \in B_R,\]
and note that
\begin{equation}
\label{lem1-3-1}
\lambda v(x)+H(x,Dv(x)) \geq \lambda \, \lambda^{-1} \, b -b = 0
\qquad \hbox{ for }\; x\in B_R.
\end{equation}

We prove that
\begin{equation}\label{lem1-3-2}
u\leq v \qquad \hbox{ in }\ B_R.
\end{equation}
By contradiction, we suppose that $\sup_{B_R}(u-v)>0$.
Since
\[
\lim_{|x|\to R^-}(u-v)(x)=-\infty,
\]
the function $u-v$ has a maximum point   at some $x_0\in B_R$
and hence, by the viscosity property of $u$
\[ \lambda u(x_0)+H(x_0,D\psi(x_0))\leq 0 \]
which yields, since $u(x_0)> v(x_0)$
\[\lambda v(x_0)+H(x_0,D\psi(x_0)) \leq 0. \]
contradicting \eqref{lem1-3-1}.

From \eqref{lem1-3-2}, we get
\[ u(x)\leq v(x)\leq \lambda^{-1} \, b+\|\psi\|_{\infty, B_{R/2}}
\qquad\hbox{ for all }\ x\in B_{R/2}. \]
This gives the assertion.
\end{proof}

\smallskip
In view of the Perron method and \eqref{A2}, we directly derive:
\begin{Theorem} \label{thm1-4}  There exists, for each
$\lambda >0$, a maximal viscosity solution $u_\lambda$ of \emph{\eqref{DP}},   which is   locally Lipschitz continuous.
\end{Theorem}

\smallskip

From now on, we denote by $u_\lambda$, for any $\lambda >0$,   the maximal
(sub)solution of \eqref{DP}.
\smallskip

\begin{Lemma}\label{lem4-2}  The  functions $u_\lambda$ are equibounded from below in $\R^N$, for $\lambda >0$
\end{Lemma}

\begin{proof}  By Proposition \ref{prescorcia} there exists a compactly supported subsolution $w$ of \eqref{EP}.
Let  $b>0$  an upper bound of $|w(x)|$ in $\R^N$,  then
 the nonpositive function $w-b$ is a
subsolution of \eqref{DP}, for any $\lambda>0$.
By the maximality of $u_\lambda $ among the subsolutions
of \eqref{DP}, we conclude that \ $ u_\lambda \geq  w-b\geq -2 \, b$ in $\R^N$.
\end{proof}
\smallskip

\def\gl{\lambda}\def\fr{\frac}\def\gd{\delta}
Here we digress slightly from the streamline and consider an example where
$N=1$ and $H(x,p)=|p|-|x|$ for $(x,p)\in\R^2$.
By solving the equations
\[
\gl u(x)+u'(x)=x \ \ \text{ and } \quad \gl u(x)-u'(x)=x \quad \text{ for }x>0,
\]
where $\gl>0$, we easily see that the functions
\[
u_+(x):=\fr{|x|}{\gl}+\fr{1}{\gl^2}(-1+e^{-\gl |x|}),
\]
and
\[
u_C(x):=\fr{|x|}{\gl}+\fr{1}{\gl^2}(1-Ce^{\gl|x|}),
\]
with $C\geq 1$, are solutions of
\begin{equation}\label{ex-81}
\gl u(x)+|u'(x)|=|x| \ \ \text{ in }\R.
\end{equation}
We can prove the following uniqueness claim: if $u$ is a solution of \eqref{ex-81}
that satisfies
\begin{equation} \label{ex-811}
\liminf_{|x|\to \infty}(u(x)+\gd e^{\gl|x|})>0 \ \ \text{ for all }\gd>0,
\end{equation}
then $u=u_+$. In particular, we have $u_\gl=u_+$ in this example. That is,
the maximal solution $u_\gl$ of \eqref{ex-81} is characterized as the unique
solution of \eqref{ex-81} that satisfies \eqref{ex-811}. This example tempts us to conjecture that, in our standing assumptions, the maximal solution $u_\gl$ is a
``unique'' solution of \eqref{DP} that is bounded from below. The authors are not able
to show the uniqueness of those solutions of \eqref{DP} that are bounded from below.

A brief idea to check the uniqueness claim above is that if $u$ is a solution of
\eqref{ex-81} and \eqref{ex-811}, then consider the function
\[
w(x):=(1-\gd)u_+(x)-\gd e^{\gl|x|} \ \ \text{ on }\R
\]
for small $\gd\in (0,\,1)$, observe that $w$ is a subsolution of \eqref{ex-81}
and
\[
\limsup_{|x|\to\infty}(w(x)-u(x))=-\infty,
\]
and apply a standard comparison theorem in a large interval $[-R,\,R]$,
to see that  $w\leq u$ in $\R$, which implies in the limit as $\gd\to 0$ that $u_+\leq u$. Observing by \eqref{ex-81} that $u(x)\leq |x|/\gl$ for all $x\in\R$, we may repeat  an argument, parallel to the above,
with $w$ and $u$ replaced by $(1-\gd)u-\gd e^{\gl|x|}$ and $u_+$, respectively,
to conclude that $u\leq u_+$.

\smallskip

\begin{Proposition}\label{propo}
The family $u_\lambda$, for $\lambda >0$, is relatively compact in $C(\R^N)$.
\end{Proposition}
\begin{proof} We already  know from  Proposition \ref{prop4-2}  and Lemma \ref{lem4-2}  that the $u_\lambda$ are locally equibounded. This implies that for any $R >0$ there exists a constant $b_R$ with
\[ H[u_\lambda] \leq b_R \qquad\hbox{in $B_R$, for any $\lambda > 0$.}\]
Taking into account  the coercivity condition \eqref{A2},  we derive from the above inequality that the $u_\lambda$ are equiLipschitz--continuous in $B_R$, for any $R >0$. This concludes the proof.
\end{proof}

\smallskip

We derive from the previous results on maximal solutions of \eqref{DP}:

\begin{Proposition}\label{square} There exists an Hamiltonian $\widetilde H$ satisfying \eqref{A1}, \eqref{A2}, \eqref{A3} plus
\begin{equation}\label{square1}
  \lim_{|p| \to + \infty} \frac {\widetilde H(x,p)}{|p|} = + \infty \qquad\hbox{for any $x \in \R^N$}
\end{equation}
such that  in addition the subsolutions of the equations \eqref{EP} and $\widetilde H[u]=0$ are the same, and
$u_\lambda$ is the maximal subsolution to $\la \, u + \widetilde H[u]=0$ for any $\la >0$.
\end{Proposition}
\begin{proof}  We set
\[b = \max_{x \in \R^N} H(x,0) \geq 0,\]
 this maximum does exist in force of \eqref{A3}.  The function  $u \equiv - \frac b \lambda$  is subsolution to
 \eqref{DP} for any $\lambda >0$, so that $\lambda \, u_\lambda \geq - b$,   and accordingly
 \begin{equation}\label{square2}
  0 \geq  \lambda \, u_\lambda + H[u_\lambda] \geq - b + H[u_\lambda].
 \end{equation}
We  define
\[ \widetilde H(x,p) = H(x,p)  + \big ( 0 \vee  ( H(x,p) - b ) \big ) ^2\]
by exploiting  the property that the square of any coercive  nonnegative  convex function from $\R^N$ to $\R$ is convex with superquadratic
growth at infinity, we see that $\widetilde H$ satisfies  \eqref{A1}, \eqref{A2}, \eqref{square1}.

Regarding  property  \eqref{A3}, we have that
\[ \min_p H(x,p) \leq H(x,0) \leq b  \qquad\hbox{for any $x \in \R^N$}\]
which implies
\begin{equation}\label{square3}
\max_{x \in \R^N} \,  \min_{p \in \R^N} H(x,p) = \max_{x \in \R^N} \,  \min_{p \in \R^N} \widetilde  H(x,p)
\end{equation}
 Moreover
\[ \max_{p \in B_\eps} H(x,p) \leq 0 \leq b \qquad\hbox{when $|x|$ is large enough}\]
so that
\begin{equation}\label{square4}
\limsup_{|x| \to + \infty} \,  \max_{p \in B_\eps} H(x,p) = \limsup_{|x| \to + \infty} \,  \max_{p \in B_\eps} \widetilde  H(x,p)
\end{equation}
where $\eps$ is the same constant appearing in \eqref{A3}.
We deduce from  \eqref{square3}, \eqref{square4} that condition \eqref{A3} holds for $\widetilde H$.
Since $b \geq 0$, we have that
\[\{(x,p) \mid H(x,p) \leq 0\} = \{(x,p) \mid \widetilde  H(x,p) \leq 0\}.\]
This implies that $0$ is the critical value for $\widetilde H$ and the equations $H[u]=0$, $\widetilde H[u]=0$ have the
same subsolutions. Further, due to $\widetilde H \geq H$, any subsolution of $\lambda \, u + \widetilde H[u]$ is also subsolution  to \eqref{DP},
which implies that the maximal subsolution to $\lambda \, u + \widetilde H[u]=0$  is less than or equal to $u_\lambda$. On the other side,
 since by \eqref{square2}
\[ H(x, Du_\lambda(x)) =  \widetilde H(x, Du_\lambda(x)) \qquad\hbox{for a.e. $x$}\]
the function $u_\lambda$ itself  is subsolution to $\lambda \, u + \widetilde H[u]=0$. This implies that $u_\lambda$ is indeed
 the maximal subsolution to $\lambda \, u + \widetilde H[u]=0$, ending the proof.

\end{proof}

The above result allows us to assume, without any loss of generality, that  superlinear growth property in \eqref{square1} holds true for $H$.
We can therefore  define via Fenchel transform the corresponding Lagrangian
\[L(x,q)= \max_p   p \cdot q - H(x,p)\]
which is  convex and coercive in $q$.   In addition, we have for any $R>0$, $x \in B_R$
\[ L(x, q) \geq R \, |q| - H(x, R \, |q|^{-1} \, q) \geq R \, |q| - \sup_{(x,p)  \in B_R \times B_R} H(x,p)\]
which shows  that
\begin{equation}\label{fenchel1}
 \lim_{|q| \to + \infty} \, \inf_{x \in B_R} \frac{L(x,q)}{|q|} = + \infty \qquad\hbox{for any $R >0$.}
\end{equation}
We moreover deduce from \eqref{A3} that there is a compact subset $K \subset \R^N$ and  positive constants
 $\de_0$, $M_0$  such that
\begin{eqnarray}
 L(x,q) &\geq& \de_0 \, |q| - H(x, q \, |q|^{-1} \, \de_0) \geq \de_0 \, |q|  \label{fenchel2}\\
  L(x,q) &\geq & - H(x,0) \geq M_0 >0  \label{fenchel13}
\end{eqnarray}
for $x \not \in K$, any $q \in \R^N$.

\bigskip
\section{Ergodic equation}

\begin{Lemma}\label{lemc} The definition  of critical value  in \eqref{defc} is well posed,  the critical value  is finite and
is actually a minimum.
\end{Lemma}
\begin{proof} By  assumption  \eqref{A3}, $H(\cdot,0)$ attains a  maximum in $\R^N$.
If $a \geq \max_{\R^N} H(x,0)$, then $H[u]=a$ admits any constant function as subsolutions.
This implies that the set in the right hand side of \eqref{defc} is nonempty.  On the other side,
if  $a < \min_p H(x,p)$ for some  $x\in \R^N$, then $H[u]=a$ does not admit any subsolution,
which shows that  the critical value  is finite. Finally it is a minimum by  standard stability properties of viscosity subsolutions.
\end{proof}

We recall that we assume throughout the paper that the critical value is $0$.

\smallskip
\begin{Proposition}\label{amarilho} There exists a solution to
\[H[u]=0  \qquad\hbox{in $\R^N$.}\]
 \end{Proposition}
 \begin{proof} As already pointed out,  there exists a subsolution to $H[u]=0$ in $\R^N$.  This implies that the
 intrinsic distance  $S_0$ is finite. We use the usual covering argument, see \cite[Theorem 3.3]{FaSi2005}  plus existence of subsolution
  going to $- \infty$ and
 Proposition \ref{prescorcia} to show there exists
  $y \in \R^N$ such that $S_0(\cdot,y)$ is a solution  to $H[u]=0$ in $\R^N$.
 \end{proof}

\medskip

 \begin{Proposition}\label{lui!} The Aubry set $\A$ is a nonempty compact subset of $\R^N$.
 \end{Proposition}

 \begin{proof} The argument of  Proposition \ref{amarilho} shows that $\A$ is nonempty, it is in addition  closed by stability properties
 of viscosity solutions. By Proposition \ref{prescorcia}, there exists a subsolution of $H[u]=0$ which is strict outside a compact subset
  $C_0 \subset \R^N$. This implies by Proposition \ref{lui!!} that $\A \subset C_0$.

 \end{proof}

 \medskip

We recall that if the ambient space is compact the ergodic equation admits solutions only at the critical level. This is not any more the case in the noncompact setting since a solutions can be found at any supercritical value as well.

\begin{Definition} We say that a solution $v$ to the critical equation is a weak KAM solution if it can be written in the form
\begin{equation}\label{defwk}
 v(x)= \min\{ v(y) + S_0(y,x) \mid y \in \A\}.
\end{equation}
\end{Definition}

We directly derive from the  definition of intrinsic distance:

\begin{Lemma}\label{enric} A function  $v$ is weak KAM solution if and only if
\[v(x) = \max \{ u(x) \mid u \;\hbox{subsolution to \eqref{EP} with $u=v$ on $\A$}\}\]
\end{Lemma}

\smallskip

We recall the following result, see for the proof  \cite{HI2008}, \cite{FaSi2005}

\begin{Lemma}\label{predebole} Let $B_0$ be a ball containing $\A$ than any solution $v$ of \eqref{EP} in $B_0$ satisfies
\[ v(x)= \min \{ v(y) + S_0(y,x) \mid y \in \partial B_0 \cup \A\}.\]
\end{Lemma}

The following characterization holds:

\begin{Theorem}\label{superKAM} A solution $v$ is weak KAM if and only if it is bounded from below.
\end{Theorem}
\begin{proof}
 Exploiting \eqref{A3},  we can find $\eps > 0$ and $R >0$   with
\[ \inf_{x \in \R^N \setminus B_R} \, \max_{p \in B_\eps} H(x,p) < 0,\]
we can further assume that $B_R \supset \A$. We consequently have
\begin{equation}\label{eiko0}
  \ell_0(\xi) \geq \eps \, \ell(\xi)
\end{equation}
for any curve $\xi$ with support contained in $\R^N \setminus B_R$. Here, to repeat,  $\ell$ stands for the Euclidean length,
 and $\ell_0$ for the intrinsic length of a curve.
 Given $x_0 \in B_R$,  we define
 \begin{equation}\label{eiko90}
   m= \min_{x \in \overline{ B_R}} S_0(x,x_0).
 \end{equation}
Assume that $v$ is not a weak KAM solution,   then  by applying Lemma \ref{predebole} to a sequence of balls with diverging radii, we find that
there exist  $x_n$ with $|x_n| \to + \infty$ such that
\begin{equation}\label{eiko001}
v(x_0)= v(x_n) + S_0(x_n,x_0).
\end{equation}
We may assume that $|x_n| > R$ for any $n$. Let $\xi_n$ be  a sequence of curves, parametrized in $[0,1]$, linking $x_n$ to $x_0$ such that
\begin{equation}\label{eiko00}
 \ell_0(\xi_n) \leq S_0(x_n,x_0) + \frac 1n.
\end{equation}
Let $t_n$ be the first entrance time of $\xi_n$ in $B_R$. This means
\[ \xi_n([0,t_n)) \cap B_R = \emptyset \quad\hbox{and} \quad  y_n:= \xi_n(t_n) \in \partial B_R.\]  We claim that
\begin{eqnarray}
 \lim_n \ell_0(\overline\xi_n) - S_0(x_n,y_n)  &=& 0  \label{eiko1}\\
\lim_n  S_0(x_n,x_0)- S_0(x_n, y_n)- S_0(y_n,x_0)&=& 0  \label{eiko2}
\end{eqnarray}
where  $\overline{\xi}_n = {\xi_n}_{ \big | ([0,t_n)}$. We in fact have by \eqref{eiko00} and the triangle inequality
\begin{eqnarray*}
 \ell_0(\xi_n) &\leq& S_0(x_n,x_0) + \frac 1n \leq S_0(x_n,y_n) + S_0(y_n,x_0) + \frac 1n  \\ &\leq&
 \ell_0 ({\xi_n}_{ \big | [t_n,1]}) +  S_0(x_n,y_n) + \frac 1n
  \leq \ell_0(\overline \xi_n) + \ell_0 ({\xi_n}_{ \big | [t_n,1]}) + \frac 1n \\ &=& \ell_0(\xi_n) + \frac 1n
\end{eqnarray*}
which implies
\begin{eqnarray*}
0 &\leq& \ell_0(\overline\xi_n) -  S_0(x_n,y_n) \leq \frac 1n \\
 0 &\leq& S_0(x_n,y_n) + S_0(y_n, x_0) - S_0(x_n,x_0) \leq \frac 1n
\end{eqnarray*}
and gives in the end  \eqref{eiko1}, \eqref{eiko2} sending $n$ to $+ \infty$.
We further have by  \eqref{eiko001}, \eqref{eiko0}, \eqref{eiko1}, \eqref{eiko2} that
\begin{eqnarray*}
  v(x_0) &=& v(x_n) + S_0(x_n,x_0)
   \geq v(x_n) + S_0(x_n,y_n)+ S_0(y_n,x_0) - \frac 1n \\ &\geq& v(x_n) + \ell_0(\overline \xi_n)  + S_0(y_n,x_0) - 2\, \frac 1n\\
   &\geq &  v(x_n) + \eps \, (|x_n| - R)  +m - 2 \, \frac 1n,
\end{eqnarray*}
where $m$ is defined as in \eqref{eiko90}, and  we finally obtain
\[v(x_n) \leq v(x_0) - m - \eps \, (|x_n| - R)  +  2 \, \frac 1n.\]
Since $|R|$ can be sent to infinity, this proves that $u$ is unbounded from below. Conversely, let $v$ be a weak KAM solution.
Let $x_1$ be a point with $|x_1| > R$. We denote by $y_0$ an element of the Aubry set  with
\[v(x_1)= v(y_0) + S_0(y_0,x_0)\]
and by  $\xi$ a curve, parametrized in $[0,1]$, linking $y_0$ to $x_1$ such that
\begin{equation}\label{eikko2}
 \ell_0(\xi) \leq S_0(y_0,x_1) + 1.
\end{equation}
Let $t_0$ be the last exit  time of $\xi$ from $B_R$. This means
\[ \xi((t_0,1)) \cap B_R = \emptyset \quad\hbox{and} \quad  z_0:= \xi(t_0) \in \partial B_R.\]
We  have by \eqref{eikko2} and the triangle inequality
\begin{eqnarray*}
 \ell_0(\xi) &\leq&  S_0(y_0,x_1) + 1 \leq S_0(y_0,z_0) + S_0(z_0,x_1) + 1 \\ &\leq& \ell ({\xi}_{ \big | [0,t_0]}) +  S_0(z_0,x_1) + 1\\
   &\leq& \ell_0 ({\xi}_{ \big | (t_0,1]}) + \ell_0 ({\xi}_{ \big | [0,t_0]}) + 1 \\  &=& \ell_0(\xi) + 1
\end{eqnarray*}
which implies
\begin{eqnarray}
0 &\leq& \ell_0(\xi_{ \big | [t_0,1]}) -  S_0(z_0,x_1) \leq 1 \label{eikko3} \\
 0 &\leq& S_0(y_0, z_0) + S_0(z_0, x_1) - S_0(y_0,x_1) \leq 1 \label{eikko4}.
\end{eqnarray}
We further have by \eqref{eiko001}, \eqref{eikko3}, \eqref{eikko4}  and the fact that $v$ is a weak KAM solution
\begin{eqnarray*}
  v(x_1) &=& v(y_0) + S_0(y_0,x_1)
   \geq v(y_0) + S_0(y_0,z_0)+ S_0(z_0,x_1) - 1  \\ &\geq& v(y_0) + \ell_0(\xi_{ \big | [t_0,1]})  + S_0(y_0,z_0) - 2\\
   &\geq &  v(z_0)    - 2 \geq  \min_{\partial B_R} v - 2
\end{eqnarray*}
which gives that $v$ is bounded from below since $x_1$ has been arbitrarily chosen  in $\R^N \setminus B_{R}$.
\end{proof}

\bigskip

\section{Localization results}\label{locale}

The results of this section will be crucial to prove that some subsets of the space of bounded functions
$\Phi: \R^N \times  \R^N$ are open in a suitable topology. This in turn will allow showing that the
existence of minimizing measures for \eqref{EP}, \eqref{DP}.

\smallskip

 \begin{Proposition}\label{palle} Any  open ball $B_0 \supset \A$ satisfies
 \[ 0= \inf \{ a \mid H=a \;\hbox{admits  subsolutions in $B_0$}\}\]
 \end{Proposition}
 \begin{proof}  Let $B_0$ be an open  ball containing the Aubry set. Assume by contradiction  that there is a
 strict subsolution to $H[u]=0$ in $B_0$. Then by standard comparison principles,  there exists one and only one solution $u_0$
 to $H[u]=0$   equal to a given trace  $g$ on $\partial B_0$ with
 \[g(x) - g(y) \leq S_0(y,x) \qquad\hbox{for any $x$, $y$ in $\partial B_0$}\]
 and it is given by
 \[u_0(x)= \min \{ g(y) + S_0(y,x) \mid y \in \partial B_0\}.\]
 We fix $z \in \A \subset B_0$, since $S_0(z, \cdot)$  is solution of $H[u]=0$ in $\R^N$, we apply the comparison principle with
 $g = S_0(z,\cdot)$ on $\partial B_0$, and  deduce that
 \[0= S_0(z,z)= S_0(z,y_0) + S_0(y_0,z) \qquad \hbox{for some $y_0 \in \partial B_0$.}\]
 This implies that we can find a sequence $\xi_n$ of cycles passing through $z$ and $y_0$ with
 \[ \ell_0(\xi_n) \to 0 \qquad\hbox{and} \qquad \inf_n \ell(\xi_n) >0.\]
 This in turn implies that $y_0 \in \A$ by  Proposition \ref{basta} , against the assumption that $\A$ is contained in the interior of $B_0$.
 \end{proof}

 \smallskip

 We proceed proving a localization property for the discounted equation.

 \begin{Proposition}\label{extra}  Given $z \in \R^N$ there exist    $\la_z >0$  such that $u_\la$ and
  the maximal subsolution of \eqref{DP}  in $C_{z,\la}$  coincide at $z$ for $\la < \la_z$ and some ball $C_{z,\la}$.
 \end{Proposition}
 \smallskip

\begin{Remark}\label{extra999}
We recall that the maximal subsolution to \eqref{DP} in some ball $B$ is nothing  but the state constraint solution.  If $u_\la(z)$
 coincide with this solution   then
 \[u_\la(z) = \inf \left \{\int_{-\infty}^0 e^{\la s} \, L(\xi, \dot \xi) \, ds \mid \xi(s) \in \overline B \; \forall\, s, \xi(0)=z \right\}\]
 If, on the contrary $u_\la(z)$  is strictly less to the maximal subsolution in $B$ then
 \[u_\la(z) = \inf \left \{ e^{- \la T} u(\xi(-T)) + \int_{-T}^0 e^{\la s} \, L(\xi, \dot \xi) \, ds \mid \xi(s) \in \overline B \, \forall \; s, \xi(0)=z,\,
  \xi(-T) \in \partial B \right \}\]
  \end{Remark}
 \smallskip

  We need two preliminary lemmata.

 \begin{Lemma}\label{lemsunday} Let $\la >0$, $x$  be an arbitrary element of $\R^N$,  $\xi$ a curve
 defined in $[-t,0]$, for some $t >0$, with $\xi(0)=x$ then
\[ u_\la(x) \leq  e^{-\la t} \, u_\la(\xi(-t)))  + \int_{-t}^0  e^{\la s} \, L(\xi, \dot \xi) \, ds.\]
If $\eta: (-\infty,0] \to \R^N$ with $\eta(0)=x$ then
\[u_\la(x) \leq \int_{-\infty}^0 e^{\la s} \, L(\eta, \dot \eta) \, ds + \liminf_{t \to + \infty} e^ {- \la t} \, u_\la(\eta(-t)).\]
\end{Lemma}
Note that $L$ and the $u_\la$ are bounded from below, the indeterminate form $+ \infty - \infty$ cannot therefore appear in the above formula.
\begin{proof}    We just prove the first part of the assertion, the second part can be obtained sending $t$ to infinity. We have
\begin{eqnarray*}
  u_\la(x) - e^{- \la t} \, u_\la(\xi(-t)) &=& \int_{-t}^0 \frac d{ds} [e^{\la s} \, u_\la(\xi(s))] \, ds  \\
   &=&  \int_{-t}^0  e^{\la s} \, ( \la \, u_\la(\xi(s)) + p(s) \cdot \dot\xi(s) ) \, ds,
\end{eqnarray*}
where $p(s)$ is a suitable element of $\partial u_\la(\xi(s))$. Taking into account that $u_\la$ is a subsolution,  we further obtain
\begin{eqnarray*}
  u_\la(x) - e^{- \la t} \, u_\la(\xi(-t)) &\leq& \int_{-t}^0  e^{\la s} \, ( \la \, u_\la(\xi(s)) + L(\xi(s), \dot\xi(s))
   + H(\xi(s), p(s))) \, ds\\
  & \leq&  \int_{-t}^0  e^{\la s} \, L(\xi, \dot\xi) \, ds.
\end{eqnarray*}
This concludes the proof.
\end{proof}

\smallskip

\begin{Lemma}\label{corparty} Given  $\eps > 0$, $\la >0$,  $x_0 \in \R^N$, let $\xi: (- T, 0] \to \R^N$, $T \in \R \cup \{+ \infty\}$,
be a curve with $\xi(0)= x_0$   and
\begin{equation}\label{lemparty1}
u_\la(x_0)  \geq \int_{- T}^ 0 e^{\la s} \, L(\xi, \dot\xi) \, ds  + \liminf_{t \to T} e^ {- \la t} \, u_\la(\xi(-t)) - \eps
\end{equation}
then
\begin{equation}\label{lemparty2}
 u_\la(\xi(- t_1)) \geq e^{-\la (t_2-t_1)} \, u_\la(\xi(-t_2)) +  \int_{t_1-t_2}^0 e^{\la s} \,  L(\xi(\cdot -t_1)
  \, \dot\xi(\cdot -t_1)) \, ds  - \eps
\end{equation}
for any $T  > t_2 > t_1 \geq 0$.
\end{Lemma}
\begin{proof}
By \eqref{lemparty1} and Lemma \ref{lemsunday} we have
\begin{eqnarray*}
  u_\la(x_0) &\geq& \int_{- T }^{-t_2} e^{\la s} L(\xi,\dot\xi) \, ds +  \int_{-t_2}^{-t_1}  e^{\la s} L(\xi,\dot\xi) \, ds \\ &+&
 \int_{-t_1}^0  e^{\la s} L(\xi,\dot\xi) \, ds + \liminf_{t \to T} e^{-\la t} \, u_\la(\xi(-t)) -\eps  \\
 &\geq & e^{- \la t_2 } u_\la(\xi(-t_2)) - \liminf_{t \to T} e^{-\la t} \, u_\la(\xi(-t))\\ &+&
 e^{-\la t_1} \,  \int_{t_1-t_2}^0 e^{\la s} \,  L(\xi(\cdot -t_1) \, \dot\xi(\cdot -t_1)) \, ds \\
  &+& u_\la(x_0) - e^{-\la t_1} \, u_\la(\xi(-t_1)) + \liminf_{t \to T} e^{-\la t} \, u_\la(\xi(-t)) - \eps
\end{eqnarray*}
which gives   \eqref{lemparty2}.
\end{proof}

\smallskip

\begin{proof}[Proof of Proposition \ref{extra}]  Given $t > 0$ and a curve $\xi$ defined in $[-t,0]$, we set to ease notations
\[\rho_\la(t,\xi)= \int_{-t}^0 e^{\la s}\, L(\xi,\dot\xi) \, ds.\]
Let $B$ be a ball containing $z$ such that there    are positive constants $M_0$, $\de_0$, $a$,  with
\begin{eqnarray}
  L(x,q) & \geq & \de_0 \, |q|  \quad\hbox{for $x  \not \in B$, any $q$}   \label{extra991}\\
  L(x,q) &\geq& M_0  \quad\hbox{for $x  \not \in B$, any $q$}  \label{extra992}\\
  u_\la(x) &\geq& -a  \quad\hbox{for any $\la >0$, $x \in \R^N$} \label{extra993}
\end{eqnarray}
see \eqref{fenchel2}, \eqref{fenchel13}, and   Lemma \ref{lem4-2} to check out that this is possible.
Further, we take $\la_z$ such that
\begin{equation}\label{extra0}
\max \{ u_\la (x) \mid x \in \partial B,\, \la < \la_z\} < \frac {M_0}\la - a - 1  .
\end{equation}
 We fix $\la < \la_z$, were the assertion not true for such a $\la$ ,    we would find, see Remark \ref{extra999}, $T_n >0$,   $x_n \in \R^N$  with $|x_n| \to + \infty$ and curves $\xi_n:[-T_n,0] \to \R^N$
 joining $x_n$ to $z$ such that
 \begin{equation}\label{extra992}
   u_\la(z) \geq  e^{-\la T_n}\,u_\la(x_n) + \rho_\la(T_n, \xi_n) - \frac 1n.
 \end{equation}
     We set for  any $n$
\[ - T'_n =  \min \{t > - T_n, \mid\xi_n(t) \in \partial B\}\]
 so that the support of $\xi_n$ is outside $B$ in the time interval $(-T_n, -T'_n)$ and $\xi_n(-T'_n) \in \partial B$.
 We  have by  \eqref{extra992} and Lemma  \ref{corparty}
 \begin{eqnarray}
  u_\la(\xi_n(- T'_n)) &\geq& e^{-\la (T'_n- T_n)} \, u_\la(\xi_n(-T_n)) + \rho_\la(T_n - T'_n,\xi_n(\cdot -T'_n))  - \frac  1n  \nonumber\\
  &\geq& - a  + \rho_\la(T_n - T'_n,\xi_n(\cdot -T'_n)) - \frac 1n,   \label{extra10}
 \end{eqnarray}
for any $n$. We first assume that $\lim_n  T_n - T'_n = + \infty$, then
\[\liminf_n  \rho_\la(T_n - T'_n,\xi_n(\cdot -T'_n))  \geq \frac {M_0}\la\]
and consequently
\begin{equation}\label{thewe10}
  \liminf_n  u_\la(\xi_n(- T'_n)) \geq  - a + \frac {M_0}\la ,
\end{equation}
 in contradiction with \eqref{extra0} since $\xi_n(-T'_n) \in \partial B$.
If instead $T_n - T'_n < T$ for any $n$, some $T >0$, we  integrate by parts, bearing in mind that $L(x,q) >0$ for $x \not\in B$, any $q$,
 to get
\begin{eqnarray*}
\nonumber  \rho_\la(T_n-T'_n,\xi_n(\cdot - T'_n)) &=&  \left [- e^{\la t} \, \int_t^0 L(\xi_n(\cdot-T'_n),\dot\xi_n(\cdot-T'_n))
 \, ds \right  ]_{T'_n - T_n}^0  \\ &+&
  \int_{T'_n - T_n}^0 \la \, e^{\la t} \,  \left ( \int_t^0 L(\xi_n(\cdot-T'_n),\dot\xi_n(\cdot-T'_n))  \,ds\right ) \, dt\\
  &\geq & e^{- \la (T_n - T'_n)} \,  \int_{T'n-T_n}^0 L(\xi(\cdot - T'_n),\dot\xi(\cdot - T'_n)) \, ds \\ &\geq&
   e^{-\la T} \, \de_0 \,| \xi_n(-T_n) - \xi_n(-T'_n)|.
\end{eqnarray*}
Taking into account that $\xi_n(-T'_n) \in \partial B$ and $|\xi_n(-T_n)| = |x_n| \to + \infty$,  we  get
\begin{equation}\label{thewe11}
 \lim_n  \rho_\la(T_n - T'_n,\xi(\cdot -T'_n))  = + \infty,
\end{equation}
 which contradicts \eqref{extra10}.
\end{proof}

\bigskip

\section{Generalized Lagrangians and narrow convergence  of measures}

We consider the space $C(\R^{2N})$ of the continuous functions from $\R^{2N}$ to $\R$. Given  such a function $\Phi(x,q)$, we say that a locally Lipschitz continuous function $u$ is a subsolution for $\Phi$ if
\[ Du(x) \cdot q \leq \Phi(x,q)  \qquad\hbox{for a.e. $x  \in \R^N$, any $q \in \R^N$}\]
we further say that $u$ is a strict subsolution  if
\[ Du(x) \cdot q \leq \Phi(x,q) - \eps  \qquad\hbox{for a.e. $x$, any $q$, some $\eps >0$.}\]
A real number $c$ is called critical value of $\Phi$ if $\Phi + c$ admits subsolutions but not strict subsolutions.

Given a discount factor $\la$, we similarly say that  a locally Lipschitz continuous function $u$ is a $\la$--discounted subsolution for $\Phi$ if
\[ \la \, u(x) + Du(x) \cdot q \leq \Phi(x,q)  \qquad\hbox{for a.e. $x  \in \R^N$, any $q \in \R^N$}\]

\smallskip

\begin{Remark}
If $\Phi$ has superlinear growth when $|q|$ goes to $+ \infty$, for any $x$, then we can apply   Fenchel transform  to
 define an Hamiltonian, denoted by $H_\Phi$, which is  convex in $p$ and satisfies \eqref{A1}, \eqref{A2}.
 A subsolution corresponding to $\Phi +a $ , for some $a \in \R$, is nothing but a subsolution of the equation $H_\Phi[u]=a$.
  Finally, the critical value of $\Phi$ is equal to the critical value of $H_\Phi$.
\end{Remark}

\medskip

We will denote by $\PP$ the space of Radon probability measure on $\R^{2N}$. Given $\mu \in \PP$ and $\Phi \in C(\R^{2N})$, integrable
 with respect to $\mu$, we will write  from now on, to ease notations,  $\langle \mu, \Phi\rangle$ in place of $\int \Phi \, d \mu$.

\medskip

 We state and prove some   convergence lemmata  with respect to the narrow topology
 we will use in what follows.

\begin{Lemma}\label{presub} Let $\Phi \in C(\R^N \times \R^N)$  be bounded from below, $\mu \in \PP$, then
\[ a:= \liminf_n \langle \mu_n, \Phi\rangle \geq \langle \mu, \Phi\rangle\]
for any sequence $\mu_n$ narrowly converging to $\mu$.
\end{Lemma}
\begin{proof}  We can assume that $a$ is finite, otherwise the assertion is trivial. Given $\eps >0$, $R  >0$, we find a subsequence $\mu_{n_k}$   of $\mu_n$ with
\[a + \eps  \geq   \langle \mu_{n_k}, \Phi \rangle \geq  \langle \mu_{n_k},  \Phi \wedge R \rangle\]
for $n_k$  large enough.
Since the functions $\Phi \wedge R$, for $R >0$,  are  bounded continuous by the assumption,
we get
\[ a + \eps \geq \lim_{n_k} \langle \mu_{n_k},  \Phi \wedge R \rangle  = \langle \mu,  \Phi \wedge R \rangle\]
and consequently, letting $\eps$ going to $0$
\[ \langle \mu,  \Phi \wedge R \rangle \leq  a  \qquad\hbox {for any $R $.}\]
Since $ \Phi \wedge R$ converges monotonically to $\Phi$ as $ R \longrightarrow + \infty$, we finally obtain  by monotone convergence Theorem
\[ \langle \mu,  \Phi \rangle = \lim_{R \to + \infty} \langle \mu,  \Phi \wedge R \rangle \leq a.\]
This ends the proof.
\end{proof}
\smallskip

We recall that   $L$ is bounded from below
  in force of \eqref{fenchel1}, \eqref{fenchel2}.   We set
  \[ m = \min_{\R^N \times \R^N} L. \]

 \begin{Lemma}\label{sub} Given any real number $a$ the sublevel
 \[ \mathcal V_a := \{ \mu \in \PP \mid \langle \mu, L \rangle \leq a\}\]
 is compact in the narrow topology, provided that it is not empty.
 \end{Lemma}
\begin{proof}  Given $\eps >0$, we find  by \eqref{fenchel2} a compact subset   $K_\eps$ of  $\R^N \times \R^N$ with
\[ L(x,q) \geq \frac 1\eps  \qquad\hbox{whenever $(x,q) \not\in K_\eps$.}\]
Given $\mu \in \mathcal V_a$, we have
\[a \geq \langle \mu, L \rangle \geq m \, \mu(K_\eps) + \frac 1\eps \, \mu (K_\eps^c) \geq - |m| +
\frac 1\eps  \, \mu(K^c_\eps) , \]
where $A^c$ indicates the complement of $A$.
This implies
\[ (a + |m|) \,  \eps \geq \mu (K_\eps^c) \qquad\hbox{for $\mu \in \mathcal V_a$}\]
and shows that the measures in $\mathcal V_a$ are uniformly tight, so that $\mathcal V_a$ is conditionally
narrowly compact. Finally if $\mu_n \in \mathcal V_a $ narrowly converges to some $\mu$, we have by Lemma \ref{presub}
\[ a \geq \liminf_n \langle \mu_n , L \rangle  \geq  \langle \mu, L \rangle\]
which  concludes the proof.
\end{proof}
\smallskip

\begin{Lemma} \label{subsub} Let $\mu_n$  a sequence in $\PP$  narrowly converging to some $\mu$, with $\langle \mu_n,L\rangle $ bounded, then
\[ \langle \mu_n,\Phi \rangle \to\langle \mu,\Phi \rangle \]
whenever $\Phi$ is  continuous, supported in $K \times \R^N$, for some compact subset $K$ of $\R^N$, and with linear growth as $|q| $ goes to infinity.
\end{Lemma}
\begin{proof}   We take $\Phi$ as in the statement.
The function
\[ R \mapsto \mu(B_R \times \R^N)\]
is nondecreasing  and  so possesses countably many discontinuities.  For any of its  continuity points $R$ we have
$\mu(\partial B_R \times \R^N)=0$ and
 consequently by the Portmanteau Theorem
 \begin{equation}\label{subsub1}
  \mu_n(B_R \times \R^N)  \to \mu(B_R \times \R^N).
 \end{equation}
We can therefore find   $R >0$ satisfying \eqref{subsub1}  such that  $\spt \Phi \subset B_R \times \R^N$,
 $\mu(B_R \times \R^N) >0$, $L >0$ in $B_R^c \times \R^N$, see \eqref{fenchel2}. We consider the conditional probabilities
\[ \overline \mu_n (\cdot)=  \mu_n(\cdot \mid B_R \times \R^N), \qquad   \overline \mu (\cdot)=  \mu(\cdot \mid B_R \times \R^N).\]
Note that the $\langle \overline \mu_n, L\rangle $  are bounded and $\overline \mu_n$ narrowly converge to $\overline \mu$.
Given $\eps >0$, we find by \eqref{fenchel1} a compact  subset $C \subset \R^N$ with
\[ L(x,q) \geq \frac 1\eps \, |q| \qquad\hbox{whenever $x\in B_R$, $q \not\in C$.}\]
We set $m = \min_{\R^N \times \R^N} L$ and denote by $a$ an upper bound of $\langle \overline\mu_n,L\rangle $.  Bearing in mind that the
the $\overline \mu_n$ are supported in $B_R \times \R^N$, we have
\begin{eqnarray*}
  a &\geq& \langle \overline\mu_n,L\rangle   \geq m \, \overline\mu_n(B_R \times C) +  \int_{ (B_R \times
C)^c} L \, d \overline\mu_n \\ &\geq& - |m| + \int_{B_R \times
(\R^N \setminus C)} L \, d \overline \mu_n \\
   &\geq&- | m|  + \frac 1\eps \, \int_{B_R \times
(\R^N \setminus C)} |q| \, d \overline\mu_n \\
\end{eqnarray*}
which implies
\[ (a + |m|) \, \eps \geq  \int_{B_R \times
(\R^N \setminus C)} |q| \, d \overline\mu_n = \int_{(B_R\times
C)^c} |q| \, d \overline\mu_n. \]
We derive that the $\overline \mu_n$ are $1$--uniformly integrable and so conditionally compact with respect to the  Wasserstein distance of order $1$,
denoted by $W^1$.  Since
$\overline \mu_n$ narrowly converges to $\overline \mu$, we deduce that the convergence actually holds with respect to $W^1$.
Such a convergence can be  equivalently expressed in duality with continuous function with linear growth at infinity, and we thus  get,
taking into account that $\spt  \, \Phi \subset B_R \times \R^N$
\[ \frac 1{\mu_n(B_R \times \R^N)} \, \langle \mu_n, \Phi \rangle = \langle \overline\mu_n, \Phi \rangle \to  \langle \overline\mu, \Phi \rangle
=  \frac 1{\mu(B_R \times \R^N)} \, \langle \mu, \Phi \rangle,\]
which gives the assertion.

\end{proof}

\smallskip

 \bigskip

\section{Mather measures for the discounted equation} \label{matherdisco}

The aim of this section is to show the existence of minimizing measures related to \eqref{DP}. More precisely, we will prove:

\begin{Theorem}\label{newkey} Given $z \in \R^N$, $0 < \la < \la_z$ ( $\la_z$ as in Proposition \ref{extra}), there exists a probability measure $\mu \in \PP$ with
\begin{equation}\label{rep81}
 \langle \mu,L \rangle = \la \, u_\la(z), \quad   \langle \mu,\Phi \rangle \geq  \la \, u(z)
\end{equation}
for any $\Phi \in C(\R^{2N})$ bounded from below admitting a $\la$--discounted subsolution and any   $\la$--discounted subsolution $u$ of $\Phi$.
\end{Theorem}

\smallskip

Our strategy is to construct a suitable convex subset of $C_b(\R^N)$,  with $L \wedge M$ for $M >0$ large
in its boundary, possessing   nonempty interior,  and then to apply Proposition \ref{sepa} about the existence of nonzero   elements
 in normal cones.  The nonzero elements in the normal cone
at $L \wedge M$ are, up to change of sign and normalization,  the probability measure appearing in the previous statement.

We proceed to prove some preliminary results.

\smallskip

\begin{Lemma}\label{preana} Given  $z \in \R^N$, $a \in \R$, $\la >0$,  an open  ball $B$ containing  $z$,
any Lipschitz function $u$ in $B$ satisfying $u(z) \geq a$ and
\begin{equation} \label{preana50}
\la \, u(x) + Du(x) \cdot q \leq L(x,q) \qquad\hbox{ for a.e. $x \in B$,  $q$ with $|q| =1$}
\end{equation}
has  Lipschitz constant in $B$ bounded from above by a quantity solely depending
on  $a$, $\la$,  $L$ and the diameter  of $B$.
\end{Lemma}

\begin{proof}  We set
\[R = \max \{L(x,q) \mid x \in \overline B, |q|=1\}.\]
Given a compact convex subset $C$ of $\overline B$, we have by \eqref{preana50}
\[ Du(x) \cdot q \leq L(x,q) - \la \min_{ C} u  \qquad \hbox{for a.e. $x \in C$, any $|q|=1$}\]
and consequently
\begin{equation}\label{preana1}
\| Du\|_{\infty,C} \leq R   - \la \min_{C} u.
\end{equation}
It is then enough to show that  $ \min_{\overline B} u$ is bounded from below, when
$u$ varies among the functions satisfying the assumptions. The idea of the proof is to give estimates on balls
of small radius centered at suitable points of $B$, and then to stitch together the information arguing by induction.

 Let $y^\#$ a minimizer of $u$ in $\overline B$, we   assume  to ease notations, but
without  any loss of generality, that
\begin{equation}\label{preana11}
|z - y^\#| = \frac M{2 \la} \leq r,
\end{equation}
where $r$ is the diameter of $B$ and   $M \in \N $,  consequently
\begin{equation}\label{preana111}
M  \leq [ 2 \, \la \,r],
\end{equation}
where $[ \cdot ]$ indicates the integer part.
We consider the segment joining $z$
to $y^\#$, and  select  $M +1$ points $x_0, x_1, \cdots, x_M$ on it with $x_0=z$,
$x_M= y^\#$, and
\[|x_{k+1} - x_k| = \frac 1{2 \la} \qquad\hbox{for $k=0, \cdots M-1$,} \]

 We moreover set
\[C_k = \overline {B(x_k, 1/2 \la)} \cap \overline B \qquad\hbox{for $k= 0, \cdots, M-1$.}\]
We have, according to \eqref{preana1}
\[u(z) - \min_{C_0} u \leq  \frac 1{2 \la} \, (R -  \la \,  \min_{C_0} u)\]
so that
\[ - \frac 12 \, \min_{C_0} u \leq - a + \frac R{2 \la}\]
and consequently
\begin{equation}\label{preana2}
 u(x_1) \geq \min_{C_0} u \geq 2 \, a - \frac R \la.
\end{equation}
We proceed showing by finite  induction on $k$   that
\begin{equation}\label{preana3}
   u(x_k) \geq  \min_{C_{k-1}} u \geq 2^k \, a -  \left (\sum_{j=0}^{k-1} 2^j \right ) \, \frac R\la.
\end{equation}
The formula is valid for $k= 1$, as shown in \eqref{preana2}, we prove  that it is true for $k+1$  assuming \eqref{preana3}. We have
\[u(x_{k}) - \min_{C_k} u \leq  \frac 1{2 \la} \, (R -  \la \,  \min_{C_k} u)\]
and
\begin{eqnarray*}
 - \frac 12 \, \min_{C_k} u &\leq& - u(x_{k}) + \frac R{2 \la}\\
   &\leq& - 2^k \, a +  \left (\sum_{j=0}^{k-1} 2^j \right ) \, \frac R\la + \frac R{2 \la}.
\end{eqnarray*}
which gives
\[u(x_{k+1}) \geq \min_{C_k} u \geq 2^{k+1} \, a -  \left (\sum_{j=0}^{k} 2^j \right ) \, \frac R\la, \]
as was claimed. Taking into account that $x_M = y^\#$ and \eqref{preana11}, \eqref{preana111} we finally get
\begin{eqnarray*}
   \min_{B} u & =&  u(x_M)\emph{} \geq 2^M \, a -   \left (\sum_{j=0}^{M-1} 2^j \right ) \, \frac R\la \geq
2^M \, a -   \left (\sum_{j=0}^{M-1} 2^j \right ) \, 2 \,r \\
&\geq& 2^{2 \la r} \, (a \wedge 0) -  \left (\sum_{j=0}^{[2 \la r]} 2^j \right ) \, r.
\end{eqnarray*}
This gives the assertion.
\end{proof}
\smallskip

We exploit the above lemma to prove:

\smallskip

\begin{Proposition}\label{ana} Given  $z \in \R^N$, $\la >0$,  $a \in \R$, a ball $B$  containing  $z$, there exists $R >1$,  such that
 any Lipschitz function $u$ in $B$ satisfying $u(z) \geq a$ and
 \begin{equation}\label{ana1}
  \la \, u(x) + Du(x) \cdot q \leq L(x,q) \qquad\hbox{ for a.e. $x \in B$, $|q| \leq R$}
\end{equation}
is a subsolution to \eqref{DP} in $B$.
\end{Proposition}
\begin{proof} If \eqref{ana1} holds true for $|q|=1$ then we know from Lemma \ref{preana} that there exists  an upper bound, denoted by $M$,
of the Lipschitz constants of all functions satisfying the assumptions, and consequently
\[\max_{B} u \leq  a + r \, M,\]
where $r$ denotes the diameter of $B$.
We take $R > 1$ such that
\[ \inf_{x \in B, \, |q| >R} \,  \frac {L(x,q)}{|q|} \geq M + \la \, (a + r \, M),\]
note that this choice is possible in force of \eqref{fenchel1}.
We get
\begin{eqnarray*}
  L(x,q) & \geq& |q| \, (  M + \la \, (a + r \, M))   \\
  &\geq& Du(x) \cdot q + \la \, u(x)
\end{eqnarray*}
 for a.e. $x \in B$, any $q$ with $|q| > R$. This last relation, together with \eqref{ana1}, gives the assertion.
\end{proof}

\medskip

We fix $z \in \R^N$, $\la < \la_z$, and set $ B_0=C_{z,\la}$, see Proposition \ref{extra}.
We further denote by $R_0$ the constant provided by Proposition  \ref{ana}
in correspondence to $B_0$, $z$, $\la$, $a =u_\la(z)$.

We define in the space $C_b(\R^{2n})$  the set
$\G_{\la,z}$ made up by the  $ \Phi \in C_b(\R^{2n})$ for which there exist a  positive constant $\eps$ and a Lipschitz continuous function $u$
in $B_0$ with $u(z) \geq u_\la(z)$ and
\begin{equation}\label{glaz}
  \la u(x) + Du(x) \cdot q \leq \Phi(x,q) - \eps
\end{equation}
for a.e. $x \in B_0$, any $q$ with $|q| \leq R_0$.

\smallskip

\begin{Proposition}\label{predracu} The set $\G_{\la,z}$ is a convex  subset, open with respect to the strict topology (see Appendix  \ref{strict} for the definition of strict topology) .
\end{Proposition}
\begin{proof} The convexity property is apparent. If $\Phi_0$ satisfies \eqref{glaz},  then any $\Phi$ with
$\|\Phi - \Phi_0\|_{\infty, \overline B_0} < \frac \eps 2$ satisfies
\[\la u(x) + Du(x) \cdot q \leq \Phi_0(x,q) - \eps \leq \Phi(x,q) - \frac \eps 2\]
for a.e. $x \in B_0$, any $q$ with $|q| \leq R_0$. This shows that any such $\Phi$ belong to $\G_{\la,z}$, in addition these elements make up an open
neighborhood of $\Phi_0$  in the compact--open topology and consequently in the strict topology. This shows that $\G_{\la,z}$ is open.

\end{proof}

\smallskip

\begin{Proposition}\label{dracu} Given $\Phi \in  C(\R^{2n})$ bounded from below and possessing a $\lambda$--discounted subsolution $u$
in $\R^N$ with
$u(z) \geq u_\la(z)$, there exists $M_0$ such that $\Phi \wedge M \in \overline{\G_{\la,z}}$ for $M \geq M_0$.
\end{Proposition}
\begin{proof} We have that
\[ \la u(x) + Du(x) \cdot q \leq \Phi(x,q)\]
for a.e. $ x \in B_0$, any $q$. Since the left hand--side of the above formula is bounded
when  $|q| \leq R_0$, we find $M_0$ such that
\[ \la u(x) + Du(x) \cdot q \leq \Phi(x,q) \wedge M   \qquad\hbox{for $M \geq M_0$, $|q| \leq R_0$, a.e. $x \in B_0$.}\]
This shows that  $\Phi(x,q) \wedge M + \eps$ belongs to $\G_{\la,z}$ for any $\eps$. We therefore get the assertion because
$\Phi(x,q) \wedge M + \eps$ strictly converges to $\Phi(x,q) \wedge M$ as $\eps$ goes to $0$.
\end{proof}

\smallskip
\begin{Proposition}\label{postdracu} There  exists $M_0$ such that $L \wedge M \in \partial \G_{\la,z}$ for $M \geq M_0$.
\end{Proposition}
\begin{proof}
Since $L$ satisfies the assumptions of Proposition \ref{dracu}, we know  that  $L \wedge M \in \overline{\G_{\la,z}}$ for $M$ greater than  or equal to some
$M_0$. It is left to show that $L \wedge M$ cannot be in the interior of $\G_{\la,z}$.  In fact, if this is the case,
there are a Lipschitz continuous function $u$, with $u(z) \geq u_\la(z)$,   and $\eps >0$ such that
 \[L(x,q) - \eps  \geq (L(x,q) \wedge M) - \eps \geq \la \, u(x) + Du(x) \cdot q\]
 for a.e. $x \in B_0$, any $|q| \leq R_0$.   We then deduce from Proposition \ref{ana} that $u$ is  strict subsolution
 to \eqref{DP} in $B_0$.  This implies in view of Proposition \ref{extra}
that $u(z)=u_\la(z)$, and $u$ is subtangent to $u_\la$ at $z$. This is impossible because
\[ L(z,q) \geq p \cdot q + \la \, u(z) + \eps \qquad\hbox{for any $p \in \partial u(z)$, $q \in \R^N$,}\]
while at least for a $p_0 \in \partial u(z)$,  any $q \in \R^N$, we must have by the subtangency condition
\[L(z,q) = p_0 \cdot q + \la \, u(z).\]
\end{proof}

\smallskip

\begin{Lemma}\label{preloca} Let $M$ such that $L \wedge M \in  \overline \G_{\la,z}$,
then any nonzero element  $\mu$ in $- N_{\overline \G_{\la,z}}(L \wedge M)$
belongs to $\PP$, up to a normalization. For any such $\mu$ we have
\[ \langle \mu , L \wedge M \rangle  = \la \, u_\la(z).\]
\end{Lemma}
\begin{proof} We know that  $L \wedge M \in \partial \G_{\la,z}$ for $M$ large enough thanks to Proposition \ref{postdracu}.  Since $\G_{\la,z}$
is a convex set  with  nonempty interior  in force of Proposition \ref{predracu}, we deduce  from Proposition \ref{sepa} that
$N_{\overline \G_{\la,z}}(L \wedge M)$ contains nonzero elements.

If one of the elements $\mu$  of $- N_{\overline \G_{\la,z}}(L \wedge M)$ were not positive,
we would find  $\Phi \in C_b(\R^{2N})$, $\Phi \geq 0$ with $\langle \mu, \Phi \rangle < 0$.
This implies that  $L \wedge M +\Phi$  belongs to $ \overline \G_{z,\la}$   and
\[\langle \mu, L \wedge M +\Phi \rangle <  \langle \mu, L \rangle\]
which is in contrast with $-\mu$ belonging to the normal cone at $L \wedge M$.  This proves that $\mu$ is a probability measure,
up to a normalization.
Since   $\Phi \equiv \la \, u_\la(z)$ belongs to $\overline\G_{\la,z}$, we get
\begin{equation}\label{preloca1}
  \langle \mu,  L \wedge M  \rangle  \leq \la \, u_\la(z)
\end{equation}
Since
\[ \rho \, (L \wedge M) + (1- \rho) \la \, u_\la(z) \in \overline\G_{\la,z} \qquad\hbox{ for $\rho >1$}\]
we further get
\[ (1 - \rho) \, \big ( \langle \mu,  L \wedge M \rangle - \la \, u_\la(z) \big ) \leq 0\]
and consequently
\[ \langle \mu , L \wedge M \rangle   \geq \la \, u_\la(z).\]
This  inequality, together with  \eqref{preloca1},  completely gives the assertion.
\end{proof}

\smallskip

\begin{Corollary}\label{corloca} Let $M_0$ be such that $L \wedge M_0 \in \overline \G_{\la,z}$, then
\[  N_{\overline{\G}_{\la,z}}(L \wedge M_0) \supset   N_{\overline\G_{\la,z}}(L \wedge M)  \qquad\hbox{for any $M > M_0$.}\]
Moreover
\[ \langle \mu,L \rangle \leq M_0 \quad\hbox{for $\mu  \in - N_{\overline\G_{\la,z}}(L \wedge M) \cap \PP$, $M \geq M_0$.}\]
\end{Corollary}
\begin{proof}  Let $M > M_0$. If $\mu  \in - N_{\overline\G_{\la,z}}(L \wedge M) \cap \PP$ then by Lemma \ref{preloca} and the
very definition of normal cone, we have
\[ \langle \mu ,L \wedge M \rangle = \la \, u_\la(z) \quad\hbox {and} \quad \langle \mu , \Phi \rangle \geq  \la \, u_\la(z) \quad\hbox{for
$\Phi \in \G_{\la,z}$}.\]
This implies that $\langle \mu ,L \wedge M_0 \rangle \geq \la \, u_\la(z)$.  On the other side, since $L \wedge M \geq L \wedge M_0$, the opposite
inequality holds true as well. We deduce that
\[\langle \mu ,L \wedge M_0 \rangle = \la \, u_\la(z),\]
 which in turn implies that
$\mu  \in - N_{\overline\G_{\la,z}}(L \wedge M_0) \cap \PP$ showing the first part of the assertion.
We claim that
\[ \spt \mu \subset \{ (x,q) \mid L(x,q) \leq M_0\} =: W.\]
 If not, there should be $(x_0,q_0) \in \spt \mu$ with $L(x_0,q_0) > M_0$. There thus  should exist a neighborhood
$U$  of $(x_0,q_0)$ with  $\mu(U) >0$ and
\[ M \geq L(x,q) > M_0  \qquad \hbox{for any $(x,q) \in U$ and a some $M > M_0$,} \]
so that
\[\int_U L \wedge M  \, d\mu = \int_U L \, d\mu > \int_U L \wedge M_0  \, d\mu\]
and consequently
\[ \langle \mu ,L \wedge M \rangle >  \langle \mu ,L \wedge M_0 \rangle \]
in contrast to what shown above.  We finally have
\[ \langle \mu ,L \wedge M \rangle = \int_W L \, d\mu \leq M_0.\]

\end{proof}

\smallskip

\smallskip

\begin{Proposition}\label{loca} There is $\mu \in \PP$ such that
\[ \langle  \mu ,\Phi \rangle \geq   \langle \mu, L\rangle = \la \, u_\la(z)\]
for any $\Phi \in C(\R^{2N})$ bounded from below and admitting a $\la$--discounted subsolution $u$ with $u(z) \geq u_\la(z)$.
\end{Proposition}
\begin{proof} We consider an increasing positively diverging sequence $M_n$ with
\[ L \wedge M_n \in \partial \G_{\la,z} \qquad\hbox{for any $n$,}\]
and $\mu_n \in - N_{\overline\G_{\la,z}}(L \wedge M_n) \cap \PP$. According to Corollary \ref{corloca}, we have
\[\langle \mu, L \rangle \leq M_1.\]
This implies  by Lemma \ref{sub}  that $\mu_n$ narrowly converges  to some $\mu \in \PP$, up to subsequences.
For any fixed $j$, we have that
\[ \mu_n \in - N_{\overline\G_{\la,z}}(L  \wedge M_j)  \qquad\hbox{for  $n \geq j$}\]
and consequently
\[\la \, u_\la(z) = \langle \mu_n, L \wedge M_j \rangle  \qquad\hbox{for $n \geq j$.}\]
We deduce that
\[ \la \, u_\la(z) = \lim_n  \langle \mu_n, L \wedge M_j\rangle =  \langle \mu, L \wedge M_j\rangle\]
and we get by Monotone Convergence  theorem, sending $j$ to infinity
\[  \langle \mu, L \rangle =  \la \, u_\la(z) .\]
If $\Phi$ is an element of $\R^{2N}$ satisfying the properties in  the statement, we have by Proposition \ref{dracu} that
\[ \Phi \wedge M_n \in \overline\G_{\la,z} \qquad\hbox{for $n$ large enough}\]
then
\[ \la \, u_\la(z) \leq  \lim_j  \langle \mu_j, \Phi \wedge M_n \rangle =  \langle \mu, \Phi \wedge M_n\rangle\]
which implies
\[ \langle \mu, \Phi \rangle  \geq  \la \, u_\la(z).\]
This ends the proof.
\end{proof}

\medskip

\begin{proof}[Proof of Theorem \ref{newkey}] Given  $\Phi$, $u$  as indicated in the statement, we have that
\[ \Phi + \la \, (u_\la(z) - u(z)) \]
has a subsolution  coinciding with $ u_\la$ at $z$.
We derive from Proposition  \ref{loca} that there exists $\mu $ with
\[ \langle \mu , \Phi\rangle  + \la \, (u_\la(z) - u(z)) \geq \la \, u_\la(z)\]
and
\[  \langle \mu , L\rangle = \la \, u_\la(z).\]
This gives the assertion.
\end{proof}

\medskip

Given $z \in \R^N$, $\la < \la_z$, we call $(\la,z)$-- Mather measure, any measure $\mu$ satisfying the statement of Theorem \ref{newkey}.
We denote by $\MM_{z,\la}$ the set of such measures $\mu$.

The formula \eqref{rep81} can be seen as an analog, to Hamilton-Jacobi equations,
of the representation of solutions of linear elliptic PDE via Green's kernel
or Poisson integral. In this regard, for $\mu\in \MM_{z,\gl}$
one may call the measure $\gl^{-1}\mu$ a Green--Poisson measure associated with $(\gl,z)$.

\bigskip

\section{Mather measures for the ergodic equation}

We perform in this section a construction parallel to that of Section \ref{matherdisco}
 to show existence of Mather measures for the ergodic equation.

  \smallskip

 The main result is:

\begin{Theorem}\label{keykey} There is $\mu \in \PP$ such that
\[ \langle  \mu ,\Phi \rangle \geq   \langle \mu, L\rangle = 0\]
for any $\Phi \in C(\R^{2N})$ bounded from below and admitting a  subsolution.
\end{Theorem}

We call  Mather measure any measure satisfying the statement of Theorem \ref{keykey}. We denote by $\MM$ the set of Mather measures.
In Propositions  \ref{suppo}  and \ref{close}  we  will actually show something more, namely that any measure
$ \mu \in \MM$ is compactly supported and that the inequality $  \langle  \mu ,\Phi \rangle  \geq 0$
holds for any $\Phi$ admitting subsolution.

\smallskip

 We start by:

\begin{Proposition}\label{barra} Given a ball  $B$,
 there exists $R >0 $ such that if a locally Lipschitz function $u$ satisfies
\begin{equation}\label{barra1}
 Du(x) \cdot q \leq L(x,q) - \eps
\end{equation}
for some $\eps >0$,  a.e. $x  \in B$, any $q$ with $|q| \leq R$,  then $u$ is strict  subsolution of $H[u]=0$ in $B$.
\end{Proposition}
\begin{proof} We set $M=  \sup_{x \in B, |q|=1} L(x,q)$.   Exploiting  \eqref{fenchel1}, we can select   $R >1$ with
\begin{equation}\label{barra2}
\inf_{x \in  B, |q|>R} \frac{L(x,q)}{|q|} >M.
\end{equation}
If  \eqref{barra1} holds true  for such an $R$ then
\[ |D u(x)| = Du(x) \cdot \frac{Du(x)}{|Du(x)|} \leq L\left (x, \frac{Du(x)}{|Du(x)|} \right ) - \eps  \qquad\hbox{for a.e. $x \in B$,}\]
which  shows that $|Du(x)|\leq M - \eps  $  in $B$. This in turn implies, in combination with   \eqref{barra2}
\[Du(x) \cdot q \leq (M - \eps)  \, |q| \leq  L(x,q)  - \eps  \]
 for a.e. $x \in B$, any $q$ with $|q| > R$.
This last inequality, together  with \eqref{barra1}, gives the assertion.
\end{proof}

\smallskip

We consider the set $\G$ of  elements $\Phi \in C_b(\R^{2N}) $  such that there exist $\eps  >0 $
and  a Lipschitz continuous function $u$ in $B_0$
with
\begin{equation}\label{dispe}
  Du(x) \cdot q \leq \Phi (x,q) - \eps \qquad\hbox{for a,e, $x \in B_0$, $|q| \leq R_0$,}
\end{equation}
where $B_0$ is an open  ball  containing  $\A$, and so satisfying   Proposition \ref{palle},   and $R_0$ is the constant provided by
Proposition \ref{barra} in correspondence
to $B_0$.

\begin{Proposition} The  set $\G$ is a convex cone with vertex at $0$ open in the strict  topology.
 \end{Proposition}
 \begin{proof} The cone property of $\G$ is apparent.  Given  $\Phi_0 \in \G$  satisfying \eqref{dispe},
 we claim that
 \[ \{ \Phi \mid \| \Phi - \Phi_0\|_{\infty, K} < \eps/2 \} \subset \G,\]
 where $K = \overline B_0 \times \overline B_R$. This will prove  the assertion because the set in the left  hand--side of the above formula is an open neighborhood of $\Phi_0$
 with respect to rhe compact--open topology, and consequently with respect to the strict topology.
 For $\phi$ belonging to it, we in fact have
 \[\Phi(x,q)   \geq \Phi_0(x,q) - \frac \eps 2 \qquad\hbox{for $x \in \overline B_0 $, $q \in \overline  B_{R_0}$},\]
 then
 \[\Phi(x,q) - \frac \eps 2  \geq \Phi_0(x,q) - \eps \geq Du(x) \cdot q  \qquad\hbox{ for a.e. $x \in B_0$, any $q \in B_R$.}\]
 This shows that $\Phi \in \mathcal G$.
 \end{proof}

Arguing as in  Proposition \ref{dracu}, we also get

\begin{Proposition}\label{grigio}  Given $\Phi \in  C(\R^{2n})$ bounded from below and possessing
 a subsolution $u$
in $\R^N$, there exists $M_0$ such that $u  \wedge M \in \overline{\G}$ for $M \geq M_0$.
\end{Proposition}

\smallskip

\begin{Proposition} There  exists $M_0$ such that $L \wedge M \in \partial \G$ for $M \geq M_0$.
\end{Proposition}
\begin{proof} By Propositions \ref{amarilho} and \ref{grigio}, we have that  $ L  \wedge  M \in \overline \G$ for $M$ suitably large,
 on the other hand  $L \wedge M $ cannot be in $\G$ otherwise by Proposition  \ref{barra} $H[u]=0$ should admit a
 strict subsolution in $B_0$, which is against  Proposition \ref{palle}.
\end{proof}

We derive arguing as in Corollary \ref{corloca}

\begin{Corollary} Let $M_0$ be such that $L \wedge M_0 \in \overline \G$, then
\[  N_{\overline\G}(L \wedge M_0) \supset   N_{\overline\G}(L \wedge M)  \qquad\hbox{for any $M > M_0$.}\]
Moreover
\[ \langle \mu,L \rangle \leq M_0 \quad\hbox{for $\mu  \in - N_{\overline\G}(L \wedge M) \cap \PP$, $M \geq M_0$.}\]
\end{Corollary}

We finally get  Theorem \ref{keykey} with the same argument as in Proposition \ref{loca}.

\bigskip

\section{Properties of Mather measures}
\begin{Proposition}\label{newkeybis}  Given  $z \in \R^N$, $\la < \la_z$ we have
\[\langle \mu, \la \, \psi +  D\psi \cdot q\rangle =  \la \, \psi(z)\]
for any $(\la,z)$--Mather measure $\mu$, $\psi \in C^1(\R^N)$, constant outside a compact subset.
\end{Proposition}
\begin{proof}
We define
\begin{eqnarray*}
  \overline L_n(x,q) &=& \lambda\, \psi(x) +  D \psi(x) \cdot q  +  \frac 1n \,( L(x,q) \vee 0) \\
 \underline L_n(x,q) &=&  -\lambda \, \psi(x) -  D\psi(x) \cdot q +  \frac 1n \,( L(x,q) \vee 0).
\end{eqnarray*}
It is clear that both $\overline L_n(x,q)$, $\underline L_n(x,q)$  are bounded from below
and the  functions $ \pm \psi$ are $\la$--discounted subsolution for  $\overline L_n$, $ \underline L_n$, respectively.
  We then  derive from Theorem \ref{newkey} that
\[ \langle \mu, \overline L_n \rangle \geq \la \, \psi(z)  \qquad\hbox{and} \qquad    \langle \mu, \underline L_n \rangle  \geq - \la \,  \psi(z),\]
which implies
\[ |\langle  \mu,  D\psi \cdot q \rangle + \lambda \, \langle \mu, \psi \rangle - \la \,  \psi(z)| \leq \frac 1n \, \langle \mu,( L(x,q) \vee 0) \rangle.
\]
Taking into account that  $\langle \mu , L \vee 0\rangle$ is finite because $L \vee 0$ is a compact perturbation of $L$, we further deduce
 sending $n$ to infinity.
  \[  \langle  \mu, \la \, \psi + D \psi  \cdot q  \rangle  = \la \,  \psi(z).\]
\end{proof}

\smallskip

\begin{Proposition}\label{suppo}   Any  $\mu \in \MM$ is compactly supported. More precisely there  exists $M >0$ such that satisfying
\[  \spt \mu \subset \A \times B_M \qquad\hbox{for any $\mu \in \MM$ } \]

\end{Proposition}
\begin{proof}
Let $\mu$ be a Mather measure,  we first prove that the support of the first marginal of $\mu$, denoted by $\mu_1$,
 is contained in the Aubry set, which is compact in force of Proposition \ref{lui!}. Assume by contradiction that there exists
 $y \in \spt \, \mu_1 \setminus \A$. This means that
 \[ \mu_1(U)= \mu(U \times \R^N) >0 \quad\hbox{for any neighborhood $U$ of $y$.}\]
 By Proposition \ref{lui!!} there exists $\eps >0$, a neighborhood $U_0$  of $y$  in $\R^N$,
 and a locally Lipschitz continuous function $v: \R^N \to \R$ with
 \begin{eqnarray}
   Dv(x) \cdot q  &\leq & L(x,q)  \qquad\hbox{a.e. in $\R^N$}  \label{suppo1}\\
    Dv(x) \cdot q   &\leq & L(x,q)   - \eps \qquad\hbox{a.e. in $U_0$}  \label{suppo2}
 \end{eqnarray}
We define
\[ \overline L(x,q) = L(x,q)  - \rho(x),\]
where $\rho$ is a continuous nonnegative function  supported in $U_0$ with $\max \rho = \eps$. We derive from  \eqref{suppo1},
\eqref{suppo2}
that $ \overline L$ admits  $v$ as subsolution   and is in addition bounded from below. On the other side, we get
\[ \langle \mu,  \overline L \rangle =  \langle \mu,  L \rangle  - \int_{U_0} \rho \, d \mu_1 < 0,\]
in contrast with  the definition of Mather measure.
We have therefore found that the projection of $\spt \,\mu$  with respect to the first component  is contained
in $\A$. Let $B$ a ball in $\R^N$ containing  $\A$. We set
\begin{equation}\label{suppo3}
 R = \sup \{ |p| \mid H(x,p) \leq 0,\, x \in B\},
\end{equation}
then $R$ is a Lipschitz constant in $B$ for any subsolution to $H[u]=0$. According to \eqref{fenchel1},  we can further choose a positive constant $M$ with
\[ L(x,q) > R \, |q| \qquad\hbox{for any $x \in B$, $|q| > M-1$}\]
We claim that
\[\spt \, \mu \subset \A \times B_{M}.\]
In fact, assume for purposes of contradiction that there is $(y_0,q_0) \in \spt \, \mu$ with $y_0 \in \A$, $|q_0| \geq M$.
 We take a neighborhood $W$ of $(y_0,q_0)$
in $\R^N \times \R^N$  with $W \subset B \times \{ |q| > M -1\}$ such that
\begin{equation}\label{suppo4}
   L(x,q) > R \, |q| + \eps \qquad\hbox{for any $(x,q) \in W$, some $\eps >0$.}
\end{equation}
We proceed defining
\[\widetilde L(x,q) = L(x,q)  - \widetilde \rho(x,q),\]
where $\widetilde\rho$ is a continuous nonnegative function  supported in $W$ with $\max \widetilde\rho = \eps$. Due to
\eqref{suppo3}, \eqref{suppo4}, we see that  any subsolution for $L $ is still a subsolution
 for $\widetilde L$, and $\widetilde L$  is bounded from below. With the same computations as in the first part of the proof,
  we find that
 \[\langle \mu, \widetilde L\rangle < 0\]
 which is impossible.

\end{proof}
\smallskip

Looking back to the proof of the previous  proposition, we realize that the argument actually shows a more general property.

\begin{Corollary}\label{postsuppo} Let $\mu \in \PP$ such that $\langle \mu, L \rangle = 0$ and
$\langle \mu, \Phi \rangle\geq 0$ for all $\Phi$  admitting subsolutions such that
\begin{equation}\label{postsuppo1}
 \Phi(x,q)= L(x,q) \qquad\hbox{ in $(\R^N  \times \R^N ) \setminus K $, with $K \subset \R^{2N}$ compact.}
\end{equation}
Then $\mu$ is compactly supported.
\end{Corollary}

\smallskip

\begin{Corollary}\label{lemclose} The set $\MM$ is a nonempty compact subset of the space of Radon measures endowed with the narrow topology.
\end{Corollary}
\begin{proof}  This is a consequence of all Mather measure being supported   in the same compact, according to  Proposition  \ref{suppo}.
The  same holds true for any narrow limit $\mu$ of sequences $\mu_n$ in $\MM$, therefore
\[\langle \mu_n, \Phi \rangle \to \langle \mu, \Phi \rangle \qquad\hbox{for any $\Phi \in C(\R^{2N})$.} \]
\end{proof}

\smallskip

We say that a measure $\mu$ is closed if
\[\langle \mu,  Du \cdot q \rangle =0 \qquad\hbox{ for any $C^1$ function $u$.}\]
We say in addition that it is locally closed if the above equality holds true just for $C^1$ functions with compact support.
For a compactly supported measure the properties of being closed or locally closed are equivalent.
\smallskip

\begin{Proposition}\label{closed} All the measures  $ \mu \in \MM$ are   closed.
\end{Proposition}
\begin{proof}
Given $\mu \in \MM$,  we consider   a  $C^1$ function $\psi$ on $\R^N$,  and  set for $\eps>0$
\[ \overline L_\eps(x,q)= D\psi(x) \cdot q  +  \eps \,( L(x,q) \vee 0) \qquad  \underline L_\eps(x,q)=  - D\psi(x) \cdot q  +  \eps \, ( L(x,q) \vee 0)\]
The argument goes along the same lines as in  Proposition \ref{newkeybis}. The  functions $\pm u$  are subsolutions
corresponding to $ \overline L_\eps$, $ \underline L_\eps$, so that
\begin{equation}\label{close1}
- \eps \langle \mu,( L(x,q) \vee 0) \rangle \leq \langle  \mu, Du \cdot q \rangle \leq \eps \langle \mu,( L(x,q) \vee 0) \rangle.
\end{equation}
Since  $\langle \mu, L \vee 0 \rangle$ is finite,
and  $\eps$ is arbitrary, we  derive from \eqref{close1}
\[ \langle \mu , Du \cdot q \rangle =0.\]
\end{proof}

\smallskip

We finally get a characterization of $\MM$.

\begin{Proposition}\label{close}    The following  conditions are equivalent:
\begin{itemize}
  \item[{\bf (i)}] $\mu \in \MM$;
  \item[{\bf (ii)}] $\mu$ is locally closed and $\langle \mu, L \rangle =0$
  \item[{\bf (iii)}]  $\langle \mu, L \rangle =0$ and any $\Phi$ admitting subsolution is integrable with respect to $\mu$ with
  $\langle \mu, \Phi \rangle \geq 0 $.
  \end{itemize}
  \end{Proposition}
\begin{proof}
The implication   ${\bf (i)} \Rightarrow {\bf (ii)}$ has been already proved in Proposition \ref{closed}.
We proceed proving  ${\bf (ii)} \Rightarrow {\bf (iii)}$.  Let $\mu$ be a measure satisfying {\bf (ii)}.

 We take
$\Phi$  admitting subsolution and coinciding with $L$ outside a compact of $\R^{2N}$, namely  satisfying  \eqref{postsuppo1}, then
the critical value of $\Phi$ is less than or equal $0$, and the corresponding Hamiltonian $H_\Phi$ satisfies \eqref{A1}--\eqref{A3}.
We  can therefore apply Proposition \ref{prescorcia}  to $H_\Phi$  and find that
there is a compactly supported subsolution for $\Phi$,  say $u$.

 Given $\eps >0$,  we can   regularize $u$ obtaining a compactly supported smooth  function $\overline u$ which
 is subsolution for  $\Phi + \eps$. Exploiting that $\mu$ is locally closed, we get
  \[\langle \mu, \Phi  + \eps \rangle \geq \langle \mu, D  \overline u  \cdot q\rangle =0\]
and the positive quantity $\eps$ being arbitrary
 \[\langle \mu, \Phi  \rangle \geq 0.\]
 This implies  by Corollary \ref{postsuppo} that $\mu$    is compactly supported, and consequently any function of $ C(\R^{2N})$
  is  integrable with respect to $\mu$.  We proceed proving that $\langle \mu, \Phi  \rangle \geq 0$ for any $\Phi$ admitting subsolution.
  We denote by $B$ an open ball of $\R^N$
 such that $\spt \mu \subset B \times \R^N$.   Taken $\eps >0$  and $\Phi$ admitting a subsolution $u$,  we can regularize $u$ in some open ball containing  $B$
  obtaining  a function $\bar u$ of class $C^1$ in $B$ such that
 \[  \Phi(x,q)  + \eps  \geq  D \bar u(x) \cdot q      \qquad\hbox{ for $(x,q) \in B \times \R^N$.}\]
Exploiting that  $\spt \mu \subset B \times \R^N$ and  that $\mu$ is   closed, we therefore get
 \[\langle \mu, \Phi + \eps \rangle \geq \langle \mu, Du \cdot q \rangle =0.\]
 This proves the claim since $\eps$ has been arbitrarily chosen.  The implication   ${\bf (iii)} \Rightarrow {\bf (i)}$ is  trivial.
\end{proof}

\bigskip

\bigskip

\section{Asymptotic results}

The first asymptotic result is:

\begin{Theorem}\label{closebis} Given    $z \in \R^N$ and an  infinitesimal sequence   $\lambda_j< \la_z$, we consider a sequence
  $\mu_j \in \MM_{\la_j,z}$,  then $\mu_j$ narrowly  converges, up to subsequences,   to a probability measure  $\mu \in \MM $.
\end{Theorem}
\begin{proof}
Since the sequence $ \langle \mu_j ,L \rangle =  \lambda_j \, u_{\lambda_j}(z)$ is bounded by Proposition \ref{prop4-2}, Lemma \ref{lem4-2}, we get
 that $\mu_j$  narrowly converges to some measure $\mu$, up to subsequences,
 in force of Lemma \ref{sub}.
Let $\psi \in C_c^1$,  then by Proposition \ref{newkeybis}
\begin{equation}\label{closebis0}
 \langle \mu_j, \la_j \, \psi + D\psi \cdot q \rangle = \la_j \,  \psi(z).
\end{equation}
Since $\psi$ is compactly supported,  then
\[\langle \mu_j, \psi \rangle \to  \langle \mu , \psi \rangle \qquad\hbox{as $j \to + \infty$} \]
and by Lemma \ref{subsub}
\[\langle \mu_j, D\psi \cdot q \rangle \to  \langle \mu , D\psi \cdot q \rangle \qquad\hbox{as $j \to + \infty$.} \]
Sending $j$ to infinity,  we thus  derive  from \eqref{closebis0}  that
$  \langle \mu, D \psi \cdot q \rangle= 0$, or in other terms that
$\mu$ is locally closed.   We further deduce via regularization of a compactly supported subsolution for $L$,  which does exist by  Proposition
\ref{prescorcia}
\begin{equation}\label{closebis1}
  \langle \mu,L \rangle \geq 0.
\end{equation}
On the other side, we  have by Lemma \ref{presub}
\[0=\lim_j \lambda_j \, u_{\lambda_j}(z) =\lim_j \langle \mu_j, L  \rangle \geq \langle \mu, L  \rangle \]
so that
\[\langle \mu, L  \rangle = 0\]
This concludes the proof in force of the characterization of Mather measures provided in Proposition \ref{close}.
\end{proof}

\smallskip
We define

\begin{equation}\label{deflim}
 w(x) = \max \{ v(x) \mid v \;\hbox {subsolution to \eqref{EP} with $\langle \mu, v\rangle \leq 0  \; \forall \mu \in \MM$}\}
\end{equation}

\begin{Proposition} The function $w$ defined above is a weak KAM solution.
\end{Proposition}
\begin{proof} As maximum of subsolutions, $w$ is a subsolution to \eqref{EP}. Since all the Mather measures are supported
in $\A \times \R^N$, then $w$  is the maximum of subsolutions with a given trace on $\A$. This implies the assertion by Lemma \ref{enric}.
\end{proof}

\smallskip
\medskip

We  give an alternative  formula for $w$ using the Peierls barrier.

\begin{Theorem} \label{thm5-1}   The function $w$ defined in \eqref{deflim} coincide with
\begin{equation}\label{enric1}
 \min\{\langle\mu,P_0(\cdot,x)\rangle \mid \mu\in \MM\} \ \ \text{ for }x\in\R^n.
\end{equation}
\end{Theorem}

\begin{proof} We denote by $u$ the function defined in \eqref{enric1}.
We know that  the function $x\mapsto P_0(y,x)$ is a weak KAM solution for any $y$. By the convexity of $H(x,p)$ in
the variables $p$, we deduce that the function $x\mapsto \langle\mu,P_0(\cdot,x)\rangle$
 is a subsolution of \eqref{EP} and the same holds true for $u$.

Next, we show that $w\leq u$ in $\R^n$. Since $w$ is a subsolution o \eqref{EP}, we have
\[w(x)-w(y)\leq P_0(y,x) \ \ \text{ for all }x,y\in \R^n. \]
Integration of both sides of the above in $y$ with respect to $\mu \in \MM$ yields
\[w(x)\leq \langle\mu, w \rangle+ \langle\mu, P_0(\cdot,x)\rangle\leq \langle\mu,P_0(\cdot,x)\rangle.
\]
This shows that $u\leq w$ in $\R^n$.

Since  $ -S_0(\cdot,z)$
is a subsolution of \eqref{EP}, the function
\[y\mapsto -P_0(x,y)=\max_{z\in \A}(-S_0(x,z)-S_0(z,y))\]
is a subsolution as well. Thus, the function
$y\mapsto -P_0(x,y)+u(x)$ is a subsolution of \eqref{EP} for all $x\in \R^n$.
Integrating this function with respect to
$\mu\in\MM$, we get
\[\langle\mu,-P_0(\cdot,x)+u(x)\rangle=-\langle \mu,P_0(\cdot,x)\rangle+ \inf_{\nu\in\MM}\langle \nu,P_0(\cdot,x)\rangle\leq 0.
\]
The definition of $w$ in \eqref{deflim} guarantees that
\[w(y)\geq -P_0(y,x)+u(x) \ \ \text{ for all }x,y\in \R^n.
\]
In particular, we have in view of Lemma   \ref{lem5-2}
\[w(z)\geq u(z) \ \ \text{ for all }z\in\A.\]
Since $w$ is a weak KAM solution and $u$ a subsolution, the inequality above ensures that $u\leq w$ in $\R^n$.
Thus, we conclude that $u=w$ in $\R^n$.
\end{proof}

 We proceed proving the main result:

\begin{Theorem}\label{card}  The functions $u_\lambda $ locally uniformly converge to $w$ defined as in \eqref{deflim}/\eqref{enric1}.
\end{Theorem}

A lemma is preliminary:

\begin{Lemma}\label{lui} We have that
\[ \langle \mu , u_\lambda \rangle \leq 0  \qquad\hbox{for any  $\lambda >0$, any Mather measure $\mu $.}\]
\end{Lemma}
\begin{proof}  The function  $u_\lambda$ is a  subsolution  for $L -  \lambda \, u_\lambda$.
We then get by Proposition \ref{close}
\[0 \leq \langle \mu, L - \lambda  \, u_\lambda \rangle =   - \lambda \, \langle \mu ,  u_\lambda  \rangle.\]
 showing   the assertion.
\end{proof}

\smallskip

\begin{proof}[Proof of Theorem \ref{card}]  Let $v$  be such that  $u_{\lambda_j} \to v$ for some sequence $\lambda_j$ converging to $0$.
 We fix $z \in \R^N$ and assume $\la_j < \la_z$.   We denote by $\mu_j$ a sequence of $(\la_j,z)$--Mather measures.
Owing to Theorem  \ref{closebis}, the $\mu_j$  converge, up to subsequences,  to some probability measure $\mu \in \MM$.

We  apply Proposition \ref{prescorcia} to the function $w$  defined in \eqref{deflim} with the compact subset $K =  \A \cup \{z\}$.
We   obtain in this way a bounded  subsolution $ \bar w$ to $H[u]=0$,  coinciding with $w$ on $\A \cup \{z\}$,
which is at the same time  a $\la_j$--discounted subsolution for $L  +\lambda_j \,\bar  w$.
Since $L  +\lambda_j \,\bar  w$  is bounded from below,  we   get by Theorem \ref{newkey}
\[\langle \mu_j, L  +\lambda_j \, \bar w  \rangle \geq   \la_j \, w(z)\]
and consequently
\begin{equation}\label{card01}
  u_{\la_j} (z) +  \langle \mu_j,  \bar w   \rangle \geq    w(z).
 \end{equation}
The function $\bar w $ is a critical subsolution agreing with $w$ on the Aubry set, and so $\bar w \leq w$ on $\R^N$.
We deduce  from the definition of $w$ in \eqref{deflim}
\[ \langle \mu , \bar  w  \rangle \leq  \langle \mu , w   \rangle \leq 0,\]
and   we get  passing to the limit in \eqref{card01} as $j \to + \infty$
\begin{equation}\label{card1}
   v(z) \geq  w(z).
\end{equation}
On the other side, given any $\nu \in \MM$,  we have  by Lemma  \ref{lui}
\[  \langle \nu,  u_{\lambda_j} \rangle \leq 0\]
and, being $\nu$ compactly supported
\[  \langle \nu, u_{\lambda_j}\rangle \longrightarrow \langle\nu,  v \rangle,  \]
which gives
\[ \langle \nu , v \rangle  \leq 0.\]
This last relation and \eqref{card1}   imply,  by  the maximality of $w$, $w(z)=v(z)$.  This concludes the proof since $z$ has been chosen arbitrarily.
\end{proof}

\bigskip

\section{Mather set}\label{mather}

The (projected) Mather set $\mathcal M$ is defined as the image by the projection ($(x,q)\mapsto x$) of the set
\[\overline{\,\bigcup_{\mu \in \MM}\spt  \,\mu\,}.\]

\medskip

The main result of the section is:

\begin{Theorem}\label{thm6-4}   Let $u_0$  be a weak KAM solution of \eqref{EP}. Then
\[u_0(x)=\max\{v(x)\mid v \;\hbox{weak KAM solution with $\langle \mu, w - u_0 \rangle \leq 0  \, \forall \, \mu\in \MM$}\} \]
\end{Theorem}

Note that by Lemma \ref{enric} the right hand--side of the above formula is equal to
\[\max\{v(x)\mid v \;\hbox{subsltn to \eqref{EP}  with $\langle \mu, w - u_0 \rangle \leq 0  \, \forall \, \mu\in \MM$}\} .\]
By the very definition of $\MM$, we have
\[\int_{\R^{2N}} \psi(x) \, d\mu =\int_{\M\times \R^N}\psi(x)\, d\mu \ \ \text{ for }\psi\in C(\R^N), \mu \in \MM.
\]
Accordingly, Theorem \ref{thm6-4} readily yields the following proposition.

\begin{Corollary}
Let $v,\,w$ be weak KAM solutions of \eqref{EP}.  Assume that $v\leq w$ in $\M$, then
$v\leq w$ in $\R^N$.
\end{Corollary}

Remark by Proposition \ref{suppo} that $\M\subset \A$. The corollary above
claims that $\M$ is an uniqueness set of \eqref{EP}, that is, if $v$, $w$ are two weak KAM solutions of \eqref{EP}
and $v=w$ in $\M$, then $v=w$ in $\R^N$.
see \cite{HI2008},  and  \cite{Fathi_book}, \cite{MiTr20181} for related results.

In our proof, we consider the following variation of the discount problem
\begin{equation} \label{thm6-4-1}
\la v+H(x,Dv(x))=\la u_0(x) \ \ \text{ in }\R^N,
\end{equation}
where $\la$ is a given positive constant and $u_0$ is a weak KAM solution  as in Theorem \ref{thm6-4}.
Here it is obvious that $u_0$ is a solution of \eqref{thm6-4-1}.

\begin{Lemma}\label{lem6-6}    Let $u_0$ be a weak KAM solution. Then, $u_0$ is a maximal subsolution of \eqref{thm6-4-1}.
\end{Lemma}

\begin{proof} Assume by contradiction that there is an usc subsolution  $v$  of
\eqref{thm6-4-1} with $v(x) > u_0(x)$ at some point $x$.  Since the maximum of two subsolutions is still a subsolution, we can assume
in addition that  $v \geq u_0$ in $\R^N$.  Therefore
\[H(x,Dv) \leq \la \, (u_0(x) - v(x)) \leq 0\]
so that $v$ is a subsolution to \eqref{EP} and is locally Lipschitz--continuous.  By Lemma \ref{lui!!}   we further derive that $v = u_0$ on $\A$. Since $u_0$ is a
weak  KAM solution, this implies that $u_0 \geq v$ in $\R^N$, which is contradictory.
\end{proof}

\smallskip

\begin{proof}[Proof of Theorem \ref{thm6-4}]   By  Proposition \ref{prescorcia}   there exists a subsolution $\bar u$ of \eqref{EP} coinciding with $u$ on $\A$ and constant
at infinity. By regularization we get for any $\eps >0$  a sequence $\bar u_\eps$  of $C^1$ functions satisfying
\begin{eqnarray*}
 |\bar u_\eps(x)- \bar u(x) | &< & \eps \qquad\hbox{for $x \in \R^N$.} \\
  |\bar u_\eps(x)-  u_0(x) | &< & \eps \qquad\hbox{for $x \in \A$.}
\end{eqnarray*}
Taking into account that $u_0$ is a weak KAM solution we derive from Lemma \ref{enric}
\[\bar u_\eps(x) \leq \bar u (x)+ \eps \leq u_0(x) + \eps  \qquad\hbox{for $x \in \R^N$.}\]
We consider the equations
\begin{eqnarray}
 \la \, u + H(x,Du) &=& \la  \, \bar u_\eps  \label{enri1}\\
  \la \, u + H(x, Du - D\bar u_\eps(x)) &=& 0 \label{enri2}
\end{eqnarray}
It is easy to check that $u$ is a subsolution to \eqref{enri1}  if and only if $u - \bar u_\eps$ is a subsolution of \eqref{enri2}. Since
$u_0 + \eps$ is the maximal solution of
\[ \la \, u + H(x,Du)= \la \, (u_0 + \eps)\]
by Lemma \ref{lem6-6} and  $u_0 + \eps \geq \bar u_\eps$ in $\R^N$,
we deduce that
\[ u_0 + \eps \geq u \qquad\hbox{for any subsolution to \eqref{enri1}.}\]
 We define the Lagrangian
\[L_\eps(x,q) = L(x,q) + D \bar u_\eps(x) \cdot q \]
corresponding to the Hamiltonian $ H(x, Du - D\bar u_\eps(x))$. A function $u$ is subsolution to $H[u]=a$, for any $a \in \R$, if and only if
$u - \bar u_\eps$ is subsolution to $H(x, Du + D \bar u_\eps(x)) = a$, this implies that $L$ and $L_\eps$ has both $0$ as critical value.
 In addition, Mather measures being closed,
we have that $\MM = \MM(L_\eps)$, where $\MM(L_\eps)$ indicates the Mather measures associated to $L_\eps$.  By applying Theorem \ref{card}
to $L_\eps$,  we see that the maximal subsolutions of \eqref{enri2} converge to
\begin{eqnarray*}
  &\max \{ u \; \hbox{subsln  for $L_\eps$ with $\langle \mu,  u \rangle  \leq 0 \, \forall \, \mu \in \MM$}\}& = \\
   &\max \{ u - \bar u_\eps \mid \; \hbox{$u$ subsln  for $L$ with $\langle \mu,  u \rangle \leq \langle \mu, \bar u_\eps \rangle \,
    \forall \, \mu \in \MM$}\}&
\end{eqnarray*}
We derive that
\begin{eqnarray*}
  u_0(x) + \eps  & \geq &  \max \{u(x) \; \hbox{ subsln  for $L$ with $\langle \mu,  u \rangle  \leq \langle \mu, \bar u_\eps \rangle  \, \forall \, \mu \in \MM$}\} \\
  &\geq&   \max \{u(x) \; \hbox{ subsln  for $L$ with $\langle \mu,  u \rangle\leq \langle \mu,   u_0\rangle   \, \forall \, \mu \in \MM$}\} - \eps\\
   &\geq& u_0(x) -  2 \, \eps
\end{eqnarray*}
We get   the assertion sending $\eps$ to $0$.

\end{proof}

\bigskip

\appendix

\section{Weak KAM  facts}\label{KAM}

 We define an intrinsic (semi)distance $S_0(\cdot,\cdot)$ in $\R^N$ related to the ergodic equation (see \cite{HI2008, Fathi_book} and also
\cite{FaSi2005})  via
\[S_0(x,y)=\sup\{u(y)-u(x)\mid u \;\hbox{subsolution of \eqref{EP}}\} \ \ \text{ for } x,y\in\R^N.\]
Since the family of subsolution to \eqref{EP} vanishing at some point  $y\in\R^N$,
is locally equi-Lipschitz continuous and, hence, locally uniformly bounded in $\R^N$,
the function $x\mapsto S_0(x,y)$ is well-defined as a locally Lipschitz continuous function
in $\R^N$.

Moreover, because of the stability of the viscosity properties under locally
uniform convergence, the function $x\mapsto S_0(x,y)$ is a subsolution of \eqref{EP} for any $y\in\R^N$.
It is clear that $S_0(x,x)=0$ for all $x\in\R^N$ and that
$S_0(x,y)\leq S_0(x,z)+S_0(z,y)$ for all $x,y,z\in\R^N$.
In view of the Perron method, for any $y\in\R^N$, the function $x\mapsto S_0(x,y)$ is a solution of \eqref{EP} in
$\R^N\setminus\{y\}$.

\smallskip

Due to the convexity of $H$ in $p$, it turns out that $S_0$ is the geodesic distance related to a length functional of the curves in $\R^N$.
We define
\[\si_0(x,q)= \max\{p \cdot q \mid H(x,p) \leq 0\}  \quad\hbox{for $(x,q) \in \R^N \times \R^N$}\]
and, given a (Lipschitz continuous) curve $\xi:[0,1] \to \R^N$, we set
\[\ell_0(\xi) = \int_0^1 \si_0(\xi,\dot\xi)\,dt.\]
Note that the above integral is invariant for orientation preserving change of parameter. We have
\begin{eqnarray*}
S_0(y,x) &=&   \inf\{\ell_0(\xi)\mid \xi:[0,1] \to \R^N \;\hbox{with $\xi(0)=y$, $\xi(1)=x$}\} \\
&=& \inf  \left \{ \int_0^T L(\zeta,\dot \zeta) \, dt \mid T >0,\, \zeta:[0,T] \to \R^N \;\hbox{with $\zeta(0)=y$, $\zeta(T)=x$} \right \}.
\end{eqnarray*}

\smallskip
We define the Aubry set $\A$ as
 \[\A=\{ y \in \R^N \mid S_0(y,\cdot) \; \hbox{ is a solution to $H=0$.}\}\]

 \begin{Proposition}\label{basta} An element $y \in \A$ if and only there is a sequence  $\xi_n$ of cycles based on $y$  with
   \[ \inf_n  \ell(\xi_n) > 0, \quad  \quad \lim_n \ell_0(\xi_n)=0.\]
 \end{Proposition}
 \begin{proof} One can argue as in  \cite[Lemma Proposition 5.4 -- Lemma 5.5]{FaSi2005}.

 \end{proof}

 \smallskip

 \begin{Proposition}\label{lui!!} Given $x \in \R^N$, if there is a subsolution of \eqref{EP} which is strict in some neighborhood of  $x$
  then $x \not\in \A$, conversely if $x \not\in \A$  there exists  a subsolution of \eqref{EP} which is strict in some neighborhood of $x$.
\end{Proposition}
\begin{proof} If  such a subsolution $u$ does exist for $x \in \A$, we find  by maximality properties  of $S_0(\cdot,x)$, that the
 function $u(x)+ S_0(x,\cdot)$  is supertangent to $u$ at $x$.
 Being  $S_0(\cdot,x)$ solution,  there is   $p_0 \in \partial u(x)$ (the generalized gradient of $u$ at $x$), with $H(x,p_0) \geq 0$,
   on the other side, being $u$ strict subsolution, any $p \in \partial u(x)$ satisfies $H(x,p) < 0$, which is contradictory.

Conversely, if $x \not\in \A$, then $S_0(x,\cdot)$ is not a solution to \eqref{EP} and there exists consequently a strict subtangent to
 $S_0(x,\cdot)$ at $x$ with $H(x,D\psi(x)) <0$, then the function
 \[ \min \{ S_0(x,\cdot) ,\psi + a\} \]
 is a subsolution of \eqref{EP} locally strict around $x$, for a suitable choice of $a > 0$.
\end{proof}

 \smallskip

The function $P_0$ in $\R^{2N}$ given by
\[P_0(x,y)=\min_{z\in \A}[S_0( x,z)+S_0(z,y)] \ \ \text{ for } x,y\in\R^N
\]
is called the Peierls barrier.  See \cite[Proposition 3.7.2]{CoIt1999} and \cite{CCIZ2018,  Fathi_book,DFIZ2016, AAIY2016}.

\begin{Lemma}\label{lem5-2}  For any $z\in\R^N$,
$P_0(z,z)=0$\ if and only if \  $z\in \A$.
\end{Lemma}

\begin{proof}  First of all, we examine some properties of the function $P_0$.
Since
\[
0=S_0(x,x)\leq S_0(x,y)+S_0(y,x) \ \ \text{ for all }x,y\in\R^N,
\]
we find that $P_0(x,x)\geq 0$ for all $x\in\R^N$.
Note next that if $z\in \A$, then
the function $x\mapsto S_0(z,x)$ is a weak KAM solution of \eqref{EP}. Hence, the function
$x\mapsto P_0(y,x)$ is a weak KAM solution of \eqref{EP} as well, for any
$y\in\R^N$.  We note by the triangle inequality for $S_0$ that for any $x,y\in\R^N$,
\[
S_0(x,y)\leq \min_{z\in\A}[S_0(x,z)+S_0(z,y)]=P_0(x,y).
\]

Now, we assume that $z\in\A$. We have
\[0\leq P_0(z,z)\leq S_0(z,z)+S_0(z,z)=0.\]
Hence, $P_0(z,z)=0$.

Next, assume that $S_0(z,z)=0$. We need to show that
the function $x\mapsto S_0(z,x)$ is a solution of \eqref{EP}. In fact,
since the function $x\mapsto S_0(z,x)$ is a solution of \eqref{EP} in $\R^N\setminus\{z\}$,
we only need to show that $H(z,D\psi(z))\geq 0$ for all  $C^1$ subtangent $\psi$  to $S_0(z, \cdot)$  at $z$. Indeed, such a function
is also subtangent to $P_0(z, \cdot)$ at $z$, and the sought inequality comes from $P_0(z, \cdot)$ being solution to \eqref{EP}.
 This completes the proof.
\end{proof}

\bigskip

\section{Strict topology}\label{strict}

 We denote by $C_0(\R^{2N})$, $C_c(\R^{2N})$ the space of compactly supported  and vanishing at infinity  continuous functions, respectively. We
 endow the space of continuous bounded functions in $\R^{2N}$, denoted by $C_b(\R^{2N})$,
 with the strict topology. It is is the locally convex Hausdorff topology
 defined by the family of seminorms
\[\{ \|\cdot\|_\Psi \mid \Psi \in C_0(\R^{2N})\}, \]
where
\[ \|\Phi\|_\Psi = \| \Phi \, \Psi\|_\infty \qquad\hbox{for any $\Psi \in C_b(\R^{2N})$.}\]
We recall that the compact open topology is instead given by the seminorms
\[\{ \|\cdot\|_\Psi \mid \Psi \in C_c(\R^{2N})\}. \]
It induces the local uniform convergence and a base of neighborhoods at any given $\Phi_0 \in C_b(\R^{2N})$ is given by
\[ \{\Phi \mid \| \Phi - \Phi_0\|_{\infty,K} < \eps \} \qquad \hbox{with $K$ compact subset of $\R^N$, $\eps >0$.}\]
The strict topology is stronger than the compact--open  topology since it
has a larger class of defining seminorms. Any open set for the compact--open topology is consequently an open set for the strict one. Further,
 the strict  topology is weaker than the topology induced by $\|\cdot\|_\infty$. Also recall
 that the completion of $C_c(\R^{2N})$ with respect to the norm topology is $C_0(\R^{2N})$, while it is $C_b(\R^{2N})$ in the strict topology.

The interest of introducing the strict topology is that we get in this frame a nice generalization  of Riesz representation theorem,
namely the topological dual of $C_b(\R^{2N})$  is the space of  signed
Radon measures with bounded variation, the
 normalized  positive elements are then Radon probability measures, see \cite{Bu1}.  The corresponding
 weak star  topology on the dual, namely the weakest topology for which
 \[ \mu \mapsto \int \Phi \, d\mu \]
 is continuous for any $\Phi \in C_b(\R^{2N})$ is called  the narrow topology.
 Accordingly a sequence of measures $\mu_n$  narrow converges
 to some $\mu$ if
 \[\int \Phi \, d\mu_n \to \int \Phi \, d\mu  \qquad\hbox{for any  $\Phi \in C_b(\R^{2N})$.}\]
 The matter is slippery because  the bounded signed Radon measures make up the topological dual of $C_0(\R^{2N})$ with the norm topology as well,
 but the  induced weak star topology, the so--called vague topology, is strictly weaker than the narrow topology. Regarding the dual of $C_b(R^{2N})$ with the norm topology,
 it is given by the bounded signed measures on the Stone--Cech compactification of $\R^{2N}$.

\section{Separation theorem}

 Let $X$ be a general locally convex Hausdorff space, we indicate by $X^*$ its topological dual and by
  $(\cdot, \cdot)$ the pairing between $X^*$ and $X$. Given a closed convex subset $E$ and $x \in \partial E$,
we denote by $N_E(x)$ the normal cone to $E$ at $X$, defined as
\[N_E(x) = \{p \in X^* \mid (p, y - x) \leq 0 \;\hbox{for any $y \in E$}\}.\]
Note that in contrast to what happens for finite dimensional spaces, in the infinite dimensional case $N_E(x)$
can reduce to $\{0\}$, see for instance \cite{ClB}. However we have

\begin{Proposition}\label{sepa} Let  $E$ be a closed convex subset of $X$ with nonempty interior,
 then $N_E(x)$ contains nonzero elements for any $x \in \partial E$.
\end{Proposition}

This is actually a simple consequence of the Hyperplane Separation theorem in locally convex Hausdorff spaces, see  \cite{RSB}, which can be stated as follows:

\begin{Theorem}  Ler $E$ be a  convex subset of $X$ with nonempty interior  and $y \not \in E$. There exists  $0 \neq p \in X^*$ with
\[ ( p, y)  \geq   ( p, x) \qquad\hbox{for any $x \in E$.}\]
\end{Theorem}

\smallskip

To get  Proposition \ref{sepa} it is enough to use the property that  the interior of any convex set is convex, and to apply
the Hyperplane Separation theorem to the interior of $E$ and to  any point in $\partial E$.

\end{document}